\documentclass[12pt]{amsart}

\usepackage{tikz-cd} 
\usepackage[left=25mm,right=25mm, top=25mm, bottom=25mm]{geometry}
\usepackage[english]{babel}
\usepackage[T1]{fontenc}
\usepackage[nice]{nicefrac}
\usepackage{amsmath}
\usepackage{amstext}
\usepackage{amsfonts}
\usepackage{amssymb}
\usepackage{enumitem}
\usepackage{amsthm}
\usepackage{mathtools}
\usepackage{dsfont}
\usepackage{color}
\usepackage{graphicx}
\usepackage[colorinlistoftodos]{todonotes}
\usepackage[colorlinks=true, allcolors=blue]{hyperref}
\usepackage{float}
\usepackage[colorinlistoftodos]{todonotes}
\usepackage{theoremref}
\usepackage{enumitem}
\usepackage{dirtytalk}
\usepackage{todonotes}
\usepackage{marginnote} 

\usepackage{multicol} 
\usepackage{titletoc}
\allowdisplaybreaks
\let\oldaddcontentsline\addcontentsline

\newcommand{\starttocentries}{\let\addcontentsline\oldaddcontentsline}
\newtheorem{theorem}{Theorem}[section]
\newtheorem{lemma}[theorem]{Lemma}
\newtheorem{prop}[theorem]{Proposition}
\newtheorem{cor}[theorem]{Corollary}
\newtheorem{lem}[theorem]{Lemma}

\newtheorem{question}[theorem]{Question}
\newtheorem*{cor*}{Corollary}
\newtheorem*{conjecture*}{Conjecture}
\newtheorem*{thm*}{Theorem}
\newtheorem*{lem*}{Lemma}

\newtheorem*{prop*}{Proposition}
\theoremstyle{definition}
\newtheorem{definition}[theorem]{Definition}

\newtheorem{example}[theorem]{Example}
\newtheorem*{defn*}{Definition}

\theoremstyle{remark}
\newtheorem{remark}[theorem]{Remark}

\newcommand{\M}{\mathcal{M}}

\newcommand{\supp}{\operatorname{supp}}

\newcommand{\vNx}{x}

\newcommand{\vNn}{\mathcal{N}}
\newcommand{\vNm}{\mathcal{M}}

\newcommand{\wo}{\text{wo}\text{-}}

\newcommand{\so}{\text{so}\text{-}}
\newcommand{\sostar}{\text{so}^*\text{-}}

\newcommand{\suba}{\text{SA}}

\newcommand{\E}{\mathbb{E}}

\makeatletter \def\l@subsection{\@tocline{2}{0pt}{1pc}{5pc}{}} \def\l@subsection{\@tocline{2}{0pt}{2pc}{6pc}{}} \makeatother

\title[Confineness, Recurrence and $C^*$-simplicity]{$C^*$-simplicity, confined subalgebras and operator algebraic uniform recurrence}
\author[Amrutam]{Tattwamasi Amrutam}
\address{Institute of Mathematics of the Polish Academy of Sciences, ul. Sniadeckich 8, 00-656, Warszawa, Poland}
\email{tattwamasiamrutam@gmail.com}
\author[Jiang]{Yongle Jiang}
\address{School of Mathematical Sciences, Dalian University of Technology, Dalian, 116024, China}
\email{yonglejiang@dlut.edu.cn}
\date{\today}

\begin{document}
\begin{abstract}
We introduce the notion of confined subalgebras in the context of the group von Neumann algebra. We also define Uniformly Recurrent States---an operator-algebraic analog of Uniformly Recurrent Subgroups. Using this framework, we show that a countable discrete group is $C^*$-simple if and only if it admits no non-trivial amenable confined subalgebras. This generalizes the well-known result of \cite{Ken20} that characterizes $C^*$-simplicity in terms of trivial amenable URSs.
\end{abstract}
\maketitle
\section{Introduction}
The interplay between the algebraic properties of a countable discrete group $\Gamma$ and the structural properties of its associated operator algebras---the reduced group $C^*$-algebra $C^*_r(\Gamma)$ and the group von Neumann algebra $L(\Gamma)$---is a cornerstone of modern operator algebras. Two fundamental regularity properties in this context are $C^*$-simplicity (the property that $C^*_r(\Gamma)$ is simple) and the unique trace property (the property that $C^*_r(\Gamma)$ admits a unique tracial state). Historically, understanding these properties has been deeply tied to identifying obstructions constructed from amenable subgroups.

It is a classical fact that a $C^*$-simple group admits no non-trivial amenable normal subgroups. Passing to the operator-algebraic realm, normal subgroups naturally translate to invariant von Neumann subalgebras (subalgebras invariant under the conjugation action of the group). Indeed, from~\cite{amrutam2025amenable}, we know that the presence of these strong rigidity properties (namely, the unique trace property) implies the absence of non-trivial invariant amenable von Neumann subalgebras. While the absence of invariant amenable subalgebras is a necessary condition for $C^*$-simplicity, it is not sufficient to fully capture the rich dynamical obstructions at play. This motivates the question of whether one can go beyond invariance. In this paper, our primary objective is to find a weakened dynamical condition that characterizes the $C^*$-simplicity of the group.

In a seminal paper, Kennedy \cite{Ken20} demonstrated that the absence of amenable normal subgroups is insufficient to characterize $C^*$-simplicity. Instead, one must analyze the topological dynamics of the Chabauty space of subgroups, $\mathrm{Sub}(\Gamma)$. Kennedy proved that $\Gamma$ is $C^*$-simple if and only if it admits no non-trivial amenable Uniformly Recurrent Subgroups (URS), a concept introduced by Glasner and Weiss \cite{GlasnerWeiss2015}. Closely related to the URS is the notion of a \emph{confined subgroup}---a subgroup whose conjugacy class is bounded away from the trivial subgroup in the Chabauty topology. Recently, Bader, Gelander, and Levitt \cite{bader2024spectral} provided spectral characterizations of confined subgroups, further establishing confinement as the premier dynamical relaxation of normality. Motivated by Kennedy's $C^*$-simplicity characterization, works of Le Boudec--Bon~\cite{le2018subgroup, le2022confined}, and the spectral work of Bader et al.~\cite{bader2024spectral}, we introduce notions of confinement and uniform recurrence into the von Neumann algebraic setting. 
\begin{definition}
[Confined Subalgebra]
A subalgebra $\mathcal{M}\le L(\Gamma)$ is called confined if 
\[\mathbb{C}\not\in \overline{\left\{\lambda(g)\mathcal{M}\lambda(g)^*: g\in\Gamma\right\}}^{\text{EM-topology}}=\overline{\left\{\mathcal{M}^g: g\in\Gamma\right\}}^{\text{EM-topology}}\]    
\end{definition}
Here, we equip the subalgebras of $L(\Gamma)$ with the Effros-Mar\'{e}chal (EM) topology. Our first main result is a characterization of $C^*$-simplicity in terms of non-confinement of amenable subalgebras. In particular, we prove the following. 
\begin{theorem}
\label{thm:mainconfined}
Let $\Gamma$ be a countable discrete group. $\Gamma$ is a $C^*$-simple group if and only if every amenable $\mathcal{M}\le L(\Gamma)$ is not confined.    
\end{theorem}
The other goal of this paper is to introduce and study \emph{confined von Neumann subalgebras} $\mathcal{M} \le L(\Gamma)$ and their dynamical limits, \emph{Uniformly Recurrent States} (see Definition~\ref{def:URA}). 
Transitioning these powerful dynamical concepts from groups to operator algebras presents an immediate and severe topological hurdle. In the geometric setting, the space of subgroups endowed with the Chabauty topology is compact, ensuring that orbit closures are well-behaved and minimal invariant subsystems necessarily exist. The natural analog for von Neumann algebras is the Effros-Mar\'{e}chal (EM) topology. However, the geometric intuition fundamentally breaks down in the operator-algebraic realm. Even for the most standard examples of separable tracial von Neumann algebras, the space of subalgebras fails to be compact, a structural divergence that forms our second main result.
\begin{theorem}
\label{thm:intro-noncompact}
Let $M$ be any separable tracial von Neumann algebra containing $L(\mathbb{Z})$. Then the space of von Neumann subalgebras $\mathrm{SubAlg}(M)$, equipped with the Effros-Mar\'{e}chal topology, is not compact.
\end{theorem}

This stark lack of compactness implies that a naive translation of topological dynamics to $\mathrm{SubAlg}(L(\Gamma))$ is fraught with pathologies; sequences of conjugated subalgebras can diverge entirely or yield limits that fail to capture the algebraic essence of the original system. Motivated by \cite{dudko2024character, jiang2024example}, to circumvent this lack of compactness and to study the dynamics of subalgebras, we employ a positive-definite function approach. We can canonically associate a positive definite function $\phi_{\mathcal{M}}$ to each subalgebra $\mathcal{M}$ using the projection $\mathbb{E}_{\mathcal{M}}$; we construct an equivariant map that transports our dynamical system into the pointwise compact space of normalized positive definite functions, $\mathrm{PD}_1(\Gamma)$. This compactification allows us to formally define Uniformly Recurrent States.
\begin{definition}
 We define the space of \emph{subalgebra states} $\mathcal{X}_\Gamma \subset \mathrm{PD}_1(\Gamma)$ as the pointwise closure of the positive definite functions arising from von Neumann subalgebras (here, $\phi_{\M}(g)=\tau_0(\mathbb{E}_{\M}(\lambda(g))\lambda(g)^*)$). Explicitly,
\[
    \mathcal{X}_\Gamma := \overline{\{\phi_{\mathcal{M}} : \mathcal{M} \le L(\Gamma)\}}^{\,\mathrm{ptwise}}.
\]
\end{definition}
Using this new compactified framework, we successfully generalize Kennedy's geometric characterization to a purely operator-algebraic one, establishing a complete correspondence between $C^*$-simplicity and the topological dynamics of amenable von Neumann subalgebras (Theorem~\ref{thm:mainconfined}). Furthermore, we show that this operator-algebraic state space is strictly richer than its geometric counterpart.

The remainder of the paper is organized as follows. In Section~\ref{sec:preliminaries}, we collect the necessary background on the Effros-Mar\'{e}chal topology, hypertraces, and conditional expectations. It is here that we prove that the Effros-Mar\'{e}chal space is generally non-compact (see Theorem~\ref{thm:non-compact}). Section~\ref{sec:confined} introduces confined subalgebras and establishes their basic properties; we also prove Theorem~\ref{thm:mainconfined} there. Section~\ref{sec:pdfandURA} develops the positive-definite function approach to confinement. Subsection~\ref{subsec:URA} transports the dynamics to the compact space \(\mathcal{X}_\Gamma\) of subalgebra states and defines Uniformly Recurrent States. Finally, Subsection~\ref{subsec:exoticURAs} constructs exotic Uniformly Recurrent States arising from finite-order automorphisms, showing that the class of Uniformly Recurrent States is strictly richer than the URSs.

\subsection*{Acknowledgments}
Y. J. is partially supported by the National Natural Science Foundation of China (Grant No. 12471118). Besides, this work was partially supported by the Simons Foundation grant (award no. SFI-MPS-T-Institutes-00010825) and from State Treasury funds as part of a task commissioned by the Minister of Science and Higher Education under the project ``Organization of the Simons Semesters at the Banach Center - New Energies in 2026-2028'' (agreement no. MNiSW/2025/DAP/491).

\section{Preliminaries}
\label{sec:preliminaries}
We now establish the necessary operator-algebraic framework. We begin by recalling the standard topologies that govern the behavior of bounded linear operators.

First, we recall the notions of the weak and ultraweak topologies on the set of bounded linear operators on a Hilbert space $\mathcal{H}$. The reader may refer to \cite{takekasi} for more details. The weak operator topology (abbreviated as WOT) is generated by open sets of the form 
\[\left\{T\in\mathbb{B}(\mathcal{H}): \left|\langle(T-T_0)\xi,\eta\rangle\right|<\epsilon\right\},\]
where $T_0\in \mathbb{B}(\mathcal{H}), \xi,\eta \in \mathcal{H}$ and $\epsilon>0$. The strong$^*$-topology is the topology induced by semi-norms of the form 
\[\left\{T\mapsto\sqrt{\left\|T\xi\right\|^2+\left\|T^*\xi\right\|^2}: T\in\mathbb{B}(\mathcal{H}),\xi\in\mathcal{H}\right\}.\]
We now turn to discussing the Effros-Mar\'{e}chal topology on sub-von Neumann algebras of a von Neumann algebra with separable predual (see~\cite{Effros, haagerup1998effros}). We denote this set of subalgebras by $\suba(\vNn)$. Additionally, we write the strong-$*$ operator topology on $\vNn$ as $\so^*$ and the weak-operator-topology as $\wo$.

The Effros-Mar\'{e}chal topology is defined in terms of the limit inferior and limit superior of a sequence of sub-von Neumann algebras $\vNm_n \in \suba(\vNn)$.

We set
$$\liminf\limits_{n\to \infty} \vNm_n:=\{x\in \vNn\ \vert  \ \exists (x_n)_{n\in\mathbb{N}} \in l^{\infty}(\mathbb{N}, \vNm_n) \, : \, \sostar\lim\limits_{n\to \infty}x_n=x\},
$$
and 
$$\limsup\limits_{n\to \infty} \vNm_n:=\langle \{x\in \vNn\ \vert  \ \exists (x_n)_{n\in\mathbb{N}} \in l^{\infty}(\mathbb{N}, \vNm_n) \, : \, \wo\lim\limits_{n\to \infty}x_n=x\}\rangle
$$ where $\langle\cdot\rangle$ means the von Neumann algebra generated by the set.
\begin{definition}\label{EM top}\cite[Definition 2.2]{ando2016ultraproducts} 
 We say that $\vNm_n$ converges to $\vNm$ if and only if $\liminf_{n\to\infty} \vNm_n = \limsup_{n\to\infty} \vNm_n = \vNm$.   
\end{definition}
It is worth noting that this topology gives a standard Borel structure on $\suba(\vNn)$.
We abbreviate the Effros-Mar\'{e}chal topology as EM-topology. In the case of a finite separable von Neumann algebra, this topology can be rephrased in a more convenient manner. 
\begin{prop}
\label{conditionalexpectationconvergence}\cite[Corollary~2.12]{haagerup1998effros}
Let $\vNm_n,\vNm\in\suba(\vNn)$, $n\in\mathbb{N}$, for a finite, separable von Neumann algebra $\vNn$. If  $\vNm_n\to \vNm$ in Effros-Mar\'{e}chal topology, then 
\begin{equation}\label{exp}
    \E_{\vNm_n}(\vNx) \xrightarrow[]{\text{so*}}  \E_{\vNm}(\vNx) , \ \forall \vNx \in \mathcal{N},
\end{equation}
where $\E_{\vNm}:\vNn\longrightarrow \vNm$ denotes the canonical conditional expectation.   
\end{prop}

\subsection{Hypertraces}
To ultimately bridge the geometric concept of amenability with our dynamical framework, we rely on hypertraces, which serve as the definitive algebraic witness to amenability.
Recall that for an inclusion of von Neumann algebras $\mathcal{M} \le \mathbb{B}(\mathcal{H})$, a state $\varphi$ on $\mathbb{B}(\mathcal{H})$ is called an $\mathcal{M}$-hypertrace if $\varphi(x m) = \varphi(m x)$ for all $x \in \mathbb{B}(\mathcal{H})$ and $m \in \mathcal{M}$. In the context of group von Neumann algebras $\mathcal{M} \le L(\Gamma) \le \mathbb{B}(\ell^2(\Gamma))$, the space of hypertraces restricting to the canonical trace $\tau_0$ on $\mathcal{M}$ is denoted $\mathrm{Hype}_{\tau_0}(\mathcal{M})$. It was shown in \cite{amrutam2025amenable} that $\mathcal{M}$ is an amenable von Neumann algebra if and only if $\mathrm{Hype}_{\tau_0}(\mathcal{M})$ is non-empty.

\subsection{Group von Neumann algebra} The primary geometric objects of our study are the von Neumann algebras associated with discrete groups, which provide the canonical setting for translating subgroup dynamics into operator algebras. Let $\ell^2(\Gamma)$ be the space of square summable $\mathbb{C}$-valued functions on $\Gamma$. There is a natural action $\Gamma\curvearrowright \ell^2(\Gamma)$ by left translation:
\[\lambda_g\xi(h):=\xi(g^{-1}h), \xi \in \ell^2(\Gamma), g,h \in \Gamma\]
The group von Neumann algebra $L(\Gamma)$ is generated (as a von Neumann algebra inside $\mathbb{B}(\ell^2(\Gamma))$) by the left regular representation $\lambda$ of $\Gamma$. The group von Neumann algebra $L(\Gamma)$ comes equipped with a canonical trace $\tau_0:L(\Gamma)\to\mathbb{C}$ defined by 
\[\tau_0\left(\lambda_g\right)=\left\{ \begin{array}{ll}
0 & \mbox{if $g\ne e$}\\
1 & \mbox{if $g=e$}\end{array}\right\}\]

\subsection{Conditional Expectation}
Let $\M\le L(\Gamma)$ be a subalgebra. It is well known that there is a trace preserving normal faithful canonical conditional expectation $\mathbb{E}_{\M}: L(\Gamma)\to\M$ which is unique up to $\tau_0$-invariance (see \cite[Theorem~9.1.2]{ADP}). 

When a subalgebra is conjugated by a group element, its associated conditional expectation transforms naturally. We formalize this equivariance below, as it forms the mechanical backbone for studying orbit closures.
\begin{lemma}
\thlabel{canonicalconditionalexpectation}
Let $\M\le L(\Gamma)$ be a von-Neumann subalgebra, and $\mathbb{E}_{\M}: L(\Gamma)\to\M$, the canonical conditional expectation. Then, $\mathbb{E}_{\M^s}: L(\Gamma)\to\lambda(s)\M\lambda(s)^*$ defined by \[\mathbb{E}_{\M^s}(x)= \lambda(s)\mathbb{E}_{\M}\left(\lambda(s)^*x\lambda(s)\right)\lambda(s)^*\]  
is the canonical conditional expectation onto $\lambda(s)\M\lambda(s)^*$.
\begin{proof}
For $x\in\M$, we see that  $\mathbb{E}_{\M^s}\left(\lambda(s)x\lambda(s)^*\right)=\lambda(s)\E_{\M}(x)\lambda(s)^*=\lambda(s)x\lambda(s)^*$ which in turn shows that $\mathbb{E}_{\M^s}$ is identity on $\lambda(s)\M\lambda(s)^*$. Now, for any $b_1=\lambda(s)m_1\lambda(s)^*$ and $b_2=\lambda(s)m_2\lambda(s)^*$  with $m_1,m_2\in\M$ and $x\in L(\Gamma)$, we see that
\begin{align*}
\mathbb{E}_{\M^s}\left(b_1xb_2\right)&:=\lambda(s)\mathbb{E}_{\M}\left(\lambda(s)^*\lambda(s)m_1\lambda(s)^*x\lambda(s)m_2\lambda(s)^*\lambda(s)\right)\lambda(s)^*\\&= \lambda(s)\mathbb{E}_{\M}\left(m_1\lambda(s)^*x\lambda(s)m_2\right)\lambda(s)^*\\&=\lambda(s)m_1\mathbb{E}_{\M}\left(\lambda(s)^*x\lambda(s)\right)m_2\lambda(s)^*\\&=\left(\lambda(s)m_1\lambda(s)^*\right)\lambda(s)\mathbb{E}_{\M}\left(\lambda(s)^*x\lambda(s)\right)\lambda(s)^*\left(\lambda(s)m_2\lambda(s)^*\right)\\&=b_1\E_{\M^s}(x)b_2.   
\end{align*}
For any positive element $x^*x\in L(\Gamma)$, since $\E_{\M}$ is a conditional expectation, we see that $\E_{\M}(\lambda(s)^*x^*x\lambda(s))=\E_{\M}\left((x\lambda(s))^*(x\lambda(s))\right)$ is a positive element and hence, is of the form $y^*y$ for some $y\in \M$. Therefore, $\E_{\M^s}(x^*x)=\lambda(s)y^*y\lambda(s)^*=\left(y\lambda(s)^*\right)^*\left(y\lambda(s)^*\right)\ge 0$. This shows that $\E_{\M^s}$ takes positive elements of $L(\Gamma)$ to positive elements of $\lambda(s)\M\lambda(s)^*$. Clearly, $\E_{\M^s}(\lambda(e))=\lambda(e)$. We now show that $\E_{\M^s}$ is $\tau_0$-invariant, i.e., $\tau_0\circ\E_{\M^s}=\tau_0$ which will complete the proof. Towards that end for any $x\in L(\Gamma)$, we see that
\begin{align*}
\tau_0\circ\E_{\M^s}(x)=\tau_0\left(\lambda(s)\mathbb{E}_{\M}\left(\lambda(s)^*x\lambda(s)\right)\lambda(s)^*\right)= \tau_0\left(\mathbb{E}_{\M}\left(\lambda(s)^*x\lambda(s)\right)\right)   
\end{align*}
Since $\E_{\M}: L(\Gamma)\to\M$ is the canonical conditional expectation, $\tau_0\circ\E_{\M}=\tau_0$. This shows that 
\[\tau_0\circ\E_{\M^s}(x)=\tau_0\left(\mathbb{E}_{\M}\left(\lambda(s)^*x\lambda(s)\right)\right)=\tau_0\left(\lambda(s)^*x\lambda(s)\right)=\tau_0(x).\]The proof is complete.
\end{proof}
\end{lemma}

\subsection{Non-compactness of \texorpdfstring{$\text{SubAlg}(\mathcal{M})$}{} under the EM-topology}
A fundamental divergence between the geometric study of subgroups and the operator-algebraic study of subalgebras lies in the compactness of their respective state spaces. While the Chabauty topology on subgroups is compact, the Effros-Mar\'{e}chal topology exhibits severe topological pathologies.

We first illustrate this lack of compactness in the highly non-commutative setting of free group factors, demonstrating that sequences of conjugated maximal abelian subalgebras can fail to converge.
\begin{lem}
Let $B=\{\sum_{i=1}^ns_i+s_i^{-1}\}''$ be the radial von Neumann subalgebra in $L(F_n)$, where $n\geq 2$ and $\{s_1,\ldots, s_n\}$ is a set of free generators in $F_n$. Then $\{s_1^kBs_1^{-k}: k\geq 1\}$ has no convergent subnet in $\text{SubAlg}(L(F_n))$. Thus, $\text{SubAlg}(L(F_n))$ is not compact under EM-topology. Hence, for any separable tracial von Neumann algebra $M$ containing the free group factor $L(F_n)$, $\text{SubAlg}(M)$ is not compact under the EM-topology.
\end{lem}
\begin{proof}
Recall that $L^2(B)$
has an orthogonal basis $\{w_i/\|w_i\|_2\}_{i=0}^{\infty}$, where $w_i\in L(F_n)$ is the sum of all words of length $i$ in $F_n$, see e.g. \cite[\S 11.5]{SS_book}. Note that $w_i$ is self-adjoint for all $i$.

Therefore, for any $y\in L(F_n)$, we have the following holds under $\|\cdot\|_2$-norm.
\begin{align}\label{eq: formula for E_B(y) where B is the radial masa}
    E_B(y)=\sum_{i=0}^{\infty}\tau\left(\frac{yw_i}{\|w_i\|_2}\right)\frac{w_i}{\|w_i\|_2}.
\end{align} 

Assume that $s_1^{j_k}Bs_1^{-j_k}\rightarrow N\in \text{SubAlg}(L(F_n))$ under EM-topology as $k\to\infty$ for some subnet $\{j_k\}$.

Step 1: $N\subseteq L(\langle s_1\rangle)$.

To show this, it suffices to show that $\left\|E_{B}(s_1^{-j_k}gs_1^{j_k})\right\|_2\to 0$ for all $g\not\in \langle s_1\rangle$.

Indeed, assume this holds, then since $s_1^{j_k}Bs_1^{-j_k}\rightarrow N$, we deduce that 
\begin{align}\label{eq: convergence under EM for radial masa using g}
\left\|E_B(s_1^{-j_k}gs_1^{j_k})-s_1^{-j_k}E_N(g)s_1^{j_k}\right\|_2\to 0.
\end{align}
Combining it with the above, we deduce that $\left\|E_N(g)\right\|_2=\left\|s_1^{-j_k}E_N(g)s_1^{j_k}\right\|_2\to 0$. Thus $E_N(g)=0$ for all $g\not\in \langle s_1\rangle$. This shows that $N\subseteq L(\langle s_1\rangle)$.

Next we show that $\left\|E_{B}(s_1^{-j_k}gs_1^{j_k})\right\|_2\to 0$ for all $g\not\in \langle s_1\rangle$. In view of \eqref{eq: formula for E_B(y) where B is the radial masa}, we have
\begin{align*}
\|E_B(s_1^{-j_k}gs_1^{j_k})\|_2=\frac{1}{\|w_{\ell(s_1^{-j_k}gs_1^{j_k})}\|_2}=\frac{1}{\sqrt{2n(2n-1)^{\ell(s_1^{-j_k}gs_1^{j_k})}}},    
\end{align*}
where $\ell(s_1^{-j_k}gs_1^{j_k})$ denotes the word length of $s_1^{-j_k}gs_1^{j_k}$ w.r.t. the standard generating set $\{s_i^{\pm}: 1\leq i\leq n\}$.
Clearly, for any $g\not\in \langle s_1\rangle$, $\ell(s_1^{-j_k}gs_1^{j_k})\rightarrow \infty$ as $k\rightarrow \infty$.

Step 2: $N=\mathbb{C}$.

For any $x\in N$, we have 
$$\left\|E_B(s_1^{-j_k}xs_1^{j_k})-s_1^{-j_k}E_N(x)s_1^{j_k}\right\|_2\to 0.$$
Note that from Step 1, we know that $x\in N\subseteq L(\langle s_1\rangle)$, an abelian subalgebra, and hence $s_1^{-j_k}xs_1^{j_k}=x$ and $s_1^{-j_k}E_N(x)s_1^{j_k}=E_N(x)$, thus the above convergence boils down to the fact that $B\ni E_B(x)=E_N(x)=x$, thus $N\subseteq B$.
Hence $N\subseteq B\cap L(\langle s_1\rangle)=\mathbb{C}$.

Step 3: $N=\mathbb{C}$ leads to a contradiction.

To see this, take $g=s_1$ in \eqref{eq: convergence under EM for radial masa using g} and note that $E_N=\tau$, then we get that $\left\|E_B(s_1)-\tau(s_1)\right\|_2\to 0$. Note that $\tau(s_1)=0$, thus $E_B(s_1)=0$, i.e. $\langle s_1, b\rangle=0$ for any $b\in B$. However, if we take $b=\sum_{i=1}^ns_i+s_i^{-1}$, then we get $\langle s_1,b\rangle=1\neq 0$, a contradiction.

The second part follows since if $L(F_n)\subset M$, then $\text{SubAlg}(L(F_n))$ is closed in $\text{SubAlg}(M)$, say by \cite[Corollary 2.10]{haagerup1998effros}.
\end{proof}
This topological obstruction is not strictly a feature of non-amenability. Non-compactness manifests even in the most basic abelian settings, thereby allowing us to conclude that $\text{SubAlg}(N)$ is not compact for any $II_1$-factor. This is Theorem~\ref{thm:intro-noncompact} from the introduction.
\begin{theorem}
\label{thm:non-compact}
Let $M=L(\mathbb{Z})\cong L^{\infty}(\{0,1\}^{\mathbb{N}}, \mu)$, where $\mu=\mu_0^{\mathbb{N}}$ and $\mu_0$ is the equal probability measure on the two point space $\{0,1\}$, i.e. $\mu_0(\{0\})=\mu_0(\{1\})=\frac{1}{2}$. Then $\text{SubAlg}(M)$ is not compact under the EM-topology. Therefore, for any separable tracial von Neumann algebra $N$ containing $L(\mathbb{Z})$, e.g., any diffuse separable tracial von Neumann algebra $N$, $\text{SubAlg}(N)$ is not compact under the EM-topology.
\end{theorem}
\begin{proof}
For each $n\geq 1$. Let $F_n\subset X:=\{0,1\}^{\mathbb{N}}$ be defined by $x\in F_n\Leftrightarrow x_n=1$. Let $E_n=F_1\cap F_n$. Set $p_n=\chi_{E_n}$ and $q_n=\chi_{E_n^c}$. Note that $E_1=F_1$ and $p_1=\chi_{F_1}$.

Note that $\chi_{F_n}\rightarrow \frac{1}{2}$ under the w*-topology on $M$ (viewed as the dual space of $L^1(X,\mu)$); equivalently, $\chi_{F_n}\rightarrow \frac{1}{2}$ under the WOT-topology on $M$. Then $p_n=\chi_{F_1}\chi_{F_n}\rightarrow \chi_{F_1}\frac{1}{2}=\frac{1}{2}p_1$ under the WOT-topology on $M$.

Set $A_n=\mathbb{C}p_n+\mathbb{C}q_n\subset M$. Then we claim that $\{A_n: n\geq 1\}$ has no convergent subnet under the EM-topology on $\text{SubAlg}(M)$.

To see this, first note that it is routine to check that $E_{A_n}(f)=\frac{\tau(fp_n)}{\tau(p_n)}p_n+\frac{\tau(fq_n)}{\tau(q_n)}q_n$ for any $f\in M$.
Write $a=\frac{1}{2}p_1$ and $b=1-a$. Since $p_n\to \frac{1}{2}p_1$ under the WOT-topology on $M$,
we deduce that $E_{A_n}(f)\to \frac{\tau(fa)}{\tau(a)}a+\frac{\tau(fb)}{\tau(b)}b$ under the WOT-topology on $M$. 
Now, assume that $\{A_n\}$
has a subnet converging to some subalgebra $Q\subset M$. Then we get that $E_Q(f)=\frac{\tau(fa)}{\tau(a)}a+\frac{\tau(fb)}{\tau(b)}b$ for any $f\in M$.

We claim that $Q=\mathbb{C}a+\mathbb{C}b$.

Indeed, $\subseteq$ holds by the above formula for $E_Q(f)$; on the other hand, if we take $f=b-\frac{\tau(b^2)}{\tau(b)}$, then  $\tau(fb)=0$ and $E_Q(f)=\frac{\tau(ba)-\frac{\tau(b^2)\tau(a)}{\tau(b)}}{\tau(a)}a=\frac{\frac{1}{8}-\frac{\frac{5}{8}}{\frac{3}{4}}\frac{1}{4}}{\frac{1}{4}}a=\frac{-1}{3}a\in Q$, thus $a\in Q$. Similarly, take $f=a-\frac{\tau(a^2)}{\tau(a)}$, then $\tau(fa)=0$ and $E_Q(f)=\frac{-1}{3}b\in Q$ and hence $b\in Q$. Therefore, $\mathbb{C}a+\mathbb{C}b\subseteq Q\subseteq \mathbb{C}a+\mathbb{C}b$, thus $Q=\mathbb{C}a+\mathbb{C}b$.

Then $a=E_Q(a)=\frac{\tau(a^2)}{\tau(a)}a+\frac{\tau(ab)}{\tau(b)}b=\frac{1}{2}a+\frac{1}{6}b$, i.e. $b=3a$, i.e. $1-\frac{1}{2}p_1=\frac{3}{2}p_1$, i.e. $p_1=\frac{1}{2}$, a contradiction.

For the second part, if $N$ is a diffuse separable tracial von Neumann algebra, then every MASA in it is diffuse and *-isomorphic to $L(\mathbb{Z})$ \cite[Corollary 3.5.3]{SS_book}. Since $\text{SubAlg}(L(\mathbb{Z}))$
is closed in $\text{SubAlg}(N)$ by \cite[Corollary 2.10]{haagerup1998effros},
this shows that $\text{SubAlg}(N)$ is not compact under the EM-topology.
\end{proof}

\section{Confined Subalgebras and their basic properties}
\label{sec:confined}
We are now ready to introduce the central object of this paper. The following definitions provide a direct operator-algebraic translation of confined subgroups, capturing subalgebras whose conjugacy classes remain uniformly bounded away from the trivial subalgebra. Our naming convention is motivated by \cite{bader2024spectral}.
\begin{definition}[Confined Subalgebra]
A subalgebra $\mathcal{M} \le L(\Gamma)$ is called confined if 
\[\mathbb{C}\not\in \overline{\left\{\lambda(g)\mathcal{M}\lambda(g)^*: g\in\Gamma\right\}}^{\text{EM-topology}}=\overline{\left\{\mathcal{M}^g: g\in\Gamma\right\}}^{\text{EM-topology}}\]
If a subalgebra $\M\le L(\Gamma)$ is not confined, we sometimes refer to it as being \say{unconfined} or \say{non-confined} by an abuse of the terminology.
\end{definition}
In practice, verifying convergence to the trivial subalgebra is most effectively done using the $L^2$-norm and conditional expectations. We record the following operational equivalences, which will serve as our primary computational tools moving forward.
\begin{prop}
\label{prop:easyobservations}
Let $\mathcal{M}\leq L(\Gamma)$ be a von Neumann subalgebra. Let $s_i\in \Gamma$ be a sequence. Then the following are equivalent.
\begin{itemize}
    \item[(1)] $s_i\mathcal{M}s_i^{-1}\rightarrow \mathbb{C}$ in EM-topology;
    \item[(2)] $E_{s_i\mathcal{M}s_i^{-1}}(x)\rightarrow \tau(x)$ under $\|\cdot \|_2$-norm for all $x\in L(\Gamma)$, where $E: L(\Gamma)\rightarrow \mathcal{M}$ denotes the trace preserving conditional expectation.
    \item[(3)] $E_{\mathcal{M}}(s_i^{-1}xs_i)\rightarrow \tau(x)$ under $\|\cdot \|_2$-norm for all $x\in L(\Gamma)$;
    \item[(4)] $E_{\mathcal{M}}(s_i^{-1}gs_i)\rightarrow 0$ under $\|\cdot\|_2$-norm for all $e\neq g\in \Gamma$. 
\end{itemize}
\end{prop}

\begin{remark}
We emphasize that this confineness should be understood as a property for the conjugacy action $\Gamma\curvearrowright \text{SubAlg}(L(\Gamma))$, and it is not clear to us whether it is preserved in general after taking  conjugation of the subalgebra by unitaries from $L(\Gamma)$.
Instead of choosing unitaries from just the group elements,  we may consider the modified version of confined subalgebra $\mathcal{M}\leq L(\Gamma)$ by requiring 
$$\mathbb{C}\not\in \overline{\{u\mathcal{M}u^*: u\in\mathcal{U}(L(\Gamma))\}}^{\text{EM-topology}}.$$ Formulating this way, \cite[Corollary 1.2]{Popa_rims} actually shows that any $Q\leq L(\Gamma)$ is not confined in this modified version provided that $L(\Gamma)\not\prec_{L(\Gamma)}Q$ (by taking $B=L(\Gamma)$ there). 
\begin{proof}
Indeed, according to \cite[Corollary 1.2]{Popa_rims}, we know there is a unitary $u=(u_n)\in L(\Gamma)^{\omega}$ such that $uL(\Gamma)u^*\perp \mathcal{M}^{\omega}$. This implies that for any $x\in L(\Gamma)$, we have $E_{\mathcal{M}^{\omega}}(uxu^*)=\tau(x)$, i.e. $\tau(x)=\lim_{n\to\omega}E_{\mathcal{M}}(u_nxu_n^*)$ for any $x\in L(\Gamma)$.
That is, $\{n\in\mathbb{N}:~\|E_{\mathcal{M}}(u_nxu_n^*)-\tau(x)\|_2<\epsilon\}\in\omega$ for any $\epsilon>0$ and any $x\in L(\Gamma)$. We now claim that we can choose a subsequence $n_i\rightarrow \infty$ to make $\lim_{i\to\infty}E_{\mathcal{M}}(u_ixu_i^*)=\tau(x)$ under $\|\cdot\|_2$-norm for any $x\in L(\Gamma)$. To see this, write $\Gamma=\{g_k: k\geq 1\}$. It suffices to prove the claim for any $x\in\Gamma$.

For each $N\geq 1$, note that $A_N:=\{n\in\mathbb{N}:~\|E_{\mathcal{M}}(u_ng_ju_n^*)-\tau(g_j)\|_2<\frac{1}{N}~\forall~1\leq j\leq N\}\in\omega$ and hence $A_N$ is an infinite subset in $\mathbb{N}$. Clearly, $A_1\supseteq A_2\supseteq A_3\cdots$. Then  pick any $n_j\in A_j\setminus A_{j+1}$ for each $j\geq 1$ such that $\lim_{i\to\infty}n_j=\infty$.
\end{proof}
  
\end{remark}
However, the following shows that the above modified version is not a good definition to be considered as a generalization of the definition of confineness for subgroups. Indeed, for subgroups, confineness is strictly weaker than finite index for many examples, e.g., in $\Gamma=F_n$, every normal subgroup, not necessarily of finite index,  is confined. 

\begin{prop}
    Let $\Lambda<\Gamma$ be countable groups. Then $L(\Lambda)$ is confined in $L(\Gamma)$ in the above sense iff $[\Gamma: \Lambda]<\infty$.
\end{prop}
\begin{proof}
    $(\Leftarrow)$ is clear.
    $\Rightarrow$: this relies on \cite[Lemma 2.2]{CI18}. Indeed, according to this lemma, we know that $L(\Gamma)\prec_{L(\Gamma)}L(\Lambda)$ implies that there exists some $g\in \Gamma$ such that $[\Gamma: g\Lambda g^{-1}]<\infty$; equivalently, $[\Gamma: \Lambda]<\infty$. Therefore, if we assume $[\Gamma: \Lambda]=\infty$, then $L(\Gamma)\not\prec_{L(\Gamma)}L(\Lambda)$ and hence $L(\Lambda)$ is not confined in the above sense as explained above, thus we get a contradiction.
\end{proof}
Unconfinement is inherited by subalgebras, i.e., any subalgebra of an unconfined von Neumann algebra is itself unconfined.
\begin{lemma}
\label{lem:unconfined}
Let $\mathcal{N}\le \mathcal{M}\le L(\Gamma)$ be an inclusion of von Neumann algebras. If $\mathcal{M}$ is unconfined, then so is $\mathcal{N}$.
\begin{proof} Let $x\in L(\Gamma)$. Since $\mathcal{M}$ is unconfined, we can find a sequence $\{s_n\}\subset\Gamma$ such that $\lambda(s_n)\mathcal{M}\lambda(s_n)^*\to\mathbb{C}$ in EM-topology. Using Proposition~\ref{prop:easyobservations}, we see that
\[\left\|\mathbb{E}_{\mathcal{M}}(\lambda(s_n)^{-1}x\lambda(s_n))-\tau_0(x)\right\|_2\xrightarrow[]{n\to\infty}0.\]
Applying $\mathbb{E}_{\mathcal{N}}$ on both sides, we get that
\[\left\|\mathbb{E}_{\mathcal{N}}\left(\mathbb{E}_{\mathcal{M}}(\lambda(s_n)^{-1}x\lambda(s_n))\right)-\tau_0(x)\right\|_2\xrightarrow[]{n\to\infty}0.\]
Since $\mathbb{E}_{\mathcal{N}}\circ\mathbb{E}_{\mathcal{M}}=\mathbb{E}_{\mathcal{N}}$, we see that
\[\left\|\mathbb{E}_{\mathcal{N}}(\lambda(s_n)^{-1}x\lambda(s_n))-\tau_0(x)\right\|_2\xrightarrow[]{n\to\infty}0.\]
The proof follows by another application of Proposition~\ref{prop:easyobservations}.
\end{proof}
\end{lemma}
We can apply this inheritance property to natural dynamical systems to easily identify classes of non-confined subalgebras arising from stabilizers.
\begin{example}
\label{ex:nonex}
Let $\Gamma\curvearrowright X$ be a minimal action such that there exists a point $x_0\in X$ with $\Gamma_{x_0}=\{e\}$. Then, for any $x\in X$, $\Gamma_x$ is not confined. Let us prove it. Fix any $x\in X$. Since $x_0\in X$ is a free point, $\Gamma_{x_0}=\Gamma_{x_0}^0=\{e\}$, where $\Gamma_{x_0}^0$ denotes the (normal) subgroup of $\Gamma_{x_0}$ consisting of elements $g\in \Gamma_{x_0}$ such that $g$ fixes pointwise a neighborhood of $x$.  Thus $x_0$ is a point of continuity for the stabilizer map $\text{Stab}:X\to\text{Sub}(\Gamma), x\mapsto\Gamma_x$ (see \cite[Lemma~2.2]{le2018subgroup}). Since the action $\Gamma\curvearrowright X$ is minimal, we can find a sequence $\{s_n\}\subset\Gamma$ such that $s_nx\to x_0$. Since $\text{Stab}$ is continuous at $x_0$, we obtain that $\Gamma_{s_nx}=\text{Stab}(s_nx)\to \text{Stab}(x_0)=\Gamma_{x_0}=\{e\}$. This is akin to saying that $s_n\Gamma_xs_n^{-1}\to\{e\}$ in Chabauty topology. Using \cite[Proposition~4.1]{amrutam2025amenable}, we obtain that $\lambda(s_n)L(\Gamma_x)\lambda(s_n)^*\to \mathbb{C}$ in the EM-topology. Moreover, using Lemma~\ref{lem:unconfined}, we see that every subalgebra $\mathcal{N}\le L(\Gamma_x)$ is also not confined. 
\end{example}
Furthermore, we observe that finite-dimensional subalgebras lack the analytical \say{mass} required to sustain confinement inside an i.c.c.\ group.
\begin{lemma}\label{lem: finite dimensional subalgebras are not confined}
Let $\Gamma$ be an i.c.c.\ group and $\mathcal{M}\le L(\Gamma)$, a finite-dimensional non-trivial subalgebra. Then $\mathcal{M}$ is not confined. 
\begin{proof}
Let $p_1,\ldots, p_n$ be a list of central projections in $\mathcal{M}$ such that $\sum_ip_i=1$ and each $p_i\M$ is a (finite dimensional) type I factor, i.e. $p_i\mathcal{M}\cong M_{k_i}(\mathbb{C})$ for some $k_i\geq 1$. Thus $\mathcal{M}\cong \oplus_{i=1}^nM_{k_i}(\mathbb{C})$.

Denote by $\{e_{s,t}^{(i)}: 1\leq s,t\leq k_i\}$ the matrix units inside $M_{k_i}(\mathbb{C})$.
Then note that for any $y\in L(\Gamma)$, we have 
\begin{align*}
    E_{\mathcal{M}}(y)=\sum_{i=1}^np_iy_i, ~\text{where}~
    y_i=\sum_{s,t=1}^{n_i}\tau(p_iy{e_{s,t}^{(i)}}^*)e_{s,t}^{(i)}.
\end{align*}
Since $\Gamma$ is i.c.c., there is a sequence $s_m\rightarrow \infty$ such that $s_m^{-1}Fs_m\cap F'=\emptyset$ for any finite $F\subseteq \Gamma\setminus \{e\}$ and any finite set $F'\subseteq \Gamma$  as $m\rightarrow \infty$. We claim that $s_m\mathcal{M}s_m^{-1}\rightarrow \mathbb{C}$ under the EM-topology.

To show this, it suffices to check that for any $e\neq g\in \Gamma$, any $1\leq i\leq n$ and any $1\leq s,t\leq n_i$, then 
$\tau(p_i(s_m^{-1}gs_m){e_{s,t}^{(i)}}^*)\rightarrow 0$ as $m\rightarrow \infty$.

Note that
$\tau(p_i(s_m^{-1}gs_m){e_{s,t}^{(i)}}^*)=\tau(s_m^{-1}gs_me_{t,s}^{(i)}p_i)$. As an element in $L(\Gamma)$, we may approximate $e_{t,s}^{(i)}p_i$ by $x_{\ell}$ with finite $\text{supp}(x_{\ell}):=F_{\ell}\subseteq \Gamma$ under $\|\cdot\|_2$-norm. Then using the choice of $s_m$, we know that $\tau(s_m^{-1}gs_mx_{\ell})=0$ for large enough $m$, thus we have for any $\epsilon>0$, pick $\ell$ such that $\|e_{t,s}^{(i)}p_i-x_{\ell}\|_2<\epsilon$, then 
\begin{align*}
    \left|\tau(s_m^{-1}gs_me_{t,s}^{(i)}p_i)\right|&\leq
    \left|\tau(s_m^{-1}gs_me_{t,s}^{(i)}p_i)-\tau(s_m^{-1}gs_mx_{\ell})\right|+\left|\tau(s_m^{-1}gs_mx_{\ell})\right|\\
   & \leq \|e_{t,s}^{(i)}p_i-x_{\ell}\|_2+ \left|\tau(s_m^{-1}gs_mx_{\ell})\right|\\
&\leq \epsilon+\left|\tau(s_m^{-1}gs_mx_{\ell})\right|  \rightarrow \epsilon+0~(\text{as}~ m\rightarrow \infty).
\end{align*}
Since $\epsilon$ is arbitrary, this shows that $\lim_{m\to\infty}\tau(s_m^{-1}gs_me_{t,s}^{(i)}p_i)=0$ and hence finishes the proof.
\end{proof}

\end{lemma}
We show that confinement reliably passes down to finite-index von Neumann subalgebras.
\begin{lemma}\label{lem: confiness passes to finite index subalgebras}
Let $\Gamma$ be an infinite group.
Let $\mathcal{M}\leq L(\Gamma)$ be a von Neumann subalgebra such that the Pimsner-Popa index $[L(\Gamma):\mathcal{M}]$ is finite. Then $\mathcal{M}$ is a confined subalgebra in $L(\Gamma)$.   
\end{lemma}
\begin{proof}
According to \cite{pimsnerpopa}, 
we know that there is a positive number $c>0$ such that $\|E_{\mathcal{M}}(x)\|_2\geq c\|x\|_2$ for all $0\neq x\in L(\Gamma)_+$. Now assume that $\mathcal{M}$ is not confined, i.e., there is a sequence $s_i\in\Gamma$ such that $s_i\mathcal{M}s_i^{-1}\rightarrow \mathbb{C}$ under EM-topology. Then take any $0\neq x\in L(\Gamma)_+$, we get that
\begin{align*}
c\|x\|_2=c\|s_i^{-1}xs_i\|_2\leq \|E_{\mathcal{M}}(s_i^{-1}xs_i)\|_2\overset{i\to\infty}{\rightarrow} \left|\tau(x)\right|.
\end{align*}
That is, we have shown that $\left|\tau(x)\right|\geq c\|x\|_2$ for all $0\neq x\in L(\Gamma)_+$. 

Since $\Gamma$ is an infinite group, it is diffuse and thus contains a sequence of non-zero projections, say  $\{p_i\}$, such that $\lim_{i\to\infty}\tau(p_i)=0$. 
Then we deduce that $\left|\tau(p_i)\right|\geq c\|p_i\|_2=c\sqrt{\tau(p_i)}$, i.e., $\tau(p_i)\geq c^2>0$. Taking $i\to\infty$ and since $p_i$ goes to zero, we get that $0\geq c^2>0$, a contradiction.
\end{proof}

\begin{question}
Let $\Gamma$ be an i.c.c.\ group.
Let $\mathcal{N}\le \mathcal{M}\leq L(\Gamma)$ be von Neumann subalgebras such that the Pimsner-Popa index $[\mathcal{M}:\mathcal{N}]$ is finite. If $\mathcal{M}$ is a confined subalgebra in $L(\Gamma)$, then is $\mathcal{N}$ also confined?
\end{question}
Note that if both $\mathcal{N}$ and $\mathcal{M}$ are group von Neumann algebras defined using subgroups of $\Gamma$, then the above question has an affirmative answer by \cite[Lemma 7.9]{bader2024spectral}.

Note that given a subgroup $\Lambda\leq \Gamma$, $\Lambda$ is confined in $\Gamma$ if and only if $L(\Lambda)$ is confined in $L(\Gamma)$. It is clear that for any finite index subgroup $\Lambda\leq \Gamma$, $L(\Lambda)$ is confined in $L(\Gamma)$.  
While Lemma~\ref{lem: confiness passes to finite index subalgebras} shows that finite index von Neumann subalgebras inherit confinement properties, the structural correspondence between finite index subgroups and finite index subalgebras is not absolute. We demonstrate this below.
\begin{lem}
\label{lem: finite index subalgs not absorbing powers of group elements}
Let $\Gamma$ be a group containing a subgroup $\langle s,t\rangle\cong \langle s\rangle *\langle t\rangle\cong \mathbb{Z}*\frac{\mathbb{Z}}{2\mathbb{Z}}$. Let $p=e_{11}=\frac{1+t}{2}$ and $e_{22}=\frac{1-t}{2}$. Let $e_{12}, e_{21}$ be the other two matrix units in $L(\Gamma)$.  
Denote by $\mathcal{M}=\{x\in L(\Gamma): e_{12}xe_{11}=0, e_{11}xe_{21}=0, e_{11}xe_{11}=e_{12}xe_{21}\}$. Then $\mathcal{M}$ is a von Neumann subalgebra in $L(\Gamma)$ with index 4. Moreover, for any $k\geq 1$, $s^k\not\in \mathcal{M}$.
\end{lem}
\begin{proof}
First, note that $L(\Gamma)\cong pL(\Gamma)p\bar{\otimes}M_2(\mathbb{C})$ via the natural map $x\overset{\Phi}{\mapsto} \sum_{i,j=1}^2e_{1i}xe_{j1}\otimes e_{ij}$. Then  it is clear that $\Phi(\mathcal{M})=pL(\Gamma)p
\bar{\otimes} \mathbb{C}I_2$, where $I_2$ denotes the $2\times 2$-identity matrix. Therefore, $\mathcal{M}$ has index 4 in $L(\Gamma)$.

Let us now prove the moreover part. Assume that $s^k\in\mathcal{M}$ for some $k\geq 1$, then $e_{12}s^ke_{11}=0$. Hence 
\begin{align*}
    0=(e_{12}s^ke_{11})^*(e_{12}s^ke_{11})
    &=(e_{11}s^{-k}e_{21})(e_{12}s^ke_{11})\\
    &=e_{11}s^{-k}e_{22}s^ke_{11}\\
    &=\frac{1+t}{2}s^{-k}\frac{1-t}{2}s^k\frac{1+t}{2}\\
    &=\frac{1}{8}(1+2t-s^{-k}ts^k-s^{-k}ts^kt-ts^{-k}ts^k-ts^{-k}ts^kt)\neq 0.
\end{align*}
\end{proof}

\begin{lem}\label{lem: PSL(n, Z) contains Z*Z_2 as a subgroup}
Let $G_n=PSL(n, \mathbb{Z})$ for any $n\geq 2$. Then $G_n$ contains a subgroup isomorphic to $\mathbb{Z}*\frac{\mathbb{Z}}{2\mathbb{Z}}$.
\end{lem}
\begin{proof}

For $n=2$, it is well-known that $G_2=PSL(2,\mathbb{Z})\cong \frac{\mathbb{Z}}{2\mathbb{Z}}*\frac{\mathbb{Z}}{3\mathbb{Z}}$. It is clear that $G_2$ contains $\mathbb{Z}*\frac{\mathbb{Z}}{2\mathbb{Z}}$ as a subgroup, e.g. $\langle tst\rangle *\langle s\rangle \cong \mathbb{Z}*\frac{\mathbb{Z}}{2\mathbb{Z}}$, where $s$ and $t$ are any order two and respectively order three elements in $G_2$. 

Let $n=3$. Set $s=\begin{pmatrix}
-1&0&0\\
0&0&1\\
0&1&0
\end{pmatrix}, t=\begin{pmatrix}
1&0&0\\
0&1&2\\
0&0&1
\end{pmatrix}$. Note that $s, t\in G_3$, $s$ is an order two element and $t$ has infinite order. We are left to show that $\langle s, t\rangle \cong \mathbb{Z}*\frac{\mathbb{Z}}{2\mathbb{Z}}$.

To see this, let us first show that for $s':=\begin{pmatrix}
0&1\\
1&0
\end{pmatrix}$ and $t'=\begin{pmatrix}
1&2\\
0&1
\end{pmatrix}$, we have $\langle s',t'\rangle\cong \mathbb{Z}*\frac{\mathbb{Z}}{2\mathbb{Z}}$. For this, we can apply the Ping-Pong argument. 

Consider two subsets of $\mathbb{C}^2$, i.e. $X=\{\begin{pmatrix}
x\\
y
\end{pmatrix}: \left|x\right|<\left|y\right|\}$ and $Y=\{\begin{pmatrix}
x\\
y
\end{pmatrix}: \left|x\right|>\left|y\right|\}$. Observe that $s'Y\subset X$ and $(t')^nX\subset Y$ for any $n\neq 0$. Now take  any non-empty word $g$ in $s'$ and $t'$, by conjugating $g$ with $(t')^n$ for large enough $n$ if necessary, we may assume that  $g=(t')^{n_1}s'(t')^{n_2}\cdots s'(t')^{n_k}$ for some $k\geq 1$, where $n_1,\cdots, n_k$ are non-zero integers. Then pick any $x\in X$, we get that 
\begin{align*}
gx=(t')^{n_1}s'(t')^{n_2}\cdots s'(t')^{n_k}x\in (t')^{n_1}s'(t')^{n_2}\cdots s'Y\subset \cdots \subset Y.
\end{align*}
Therefore, $g\neq e$ since $X\cap Y=\emptyset$.

Now, we can finish the proof by observing that whenever we have a non-empty word $g$ in $s, t$, then we may apply the above to deduce that the right-bottom $2\times 2$ corner of $g$ is not an identity matrix, and thus $g\neq e$.

For the general case that $n\geq 4$, we may consider the matrices $s_n=\begin{pmatrix}
I_{n-3}&0\\
0&s
\end{pmatrix}$ and $t_n=\begin{pmatrix}
I_{n-3}&0\\
0&t
\end{pmatrix}$, where $I_{n-3}$ denotes the identity $(n-3)\times (n-3)$ matrix. Note that $s_n,t_n\in SL(n,\mathbb{Z})$ and we identify them as their images in $PSL(n, \mathbb{Z})$. Then the above argument still works by noting that $gx\in Y$ implies, in fact, that $g\neq [I_n]\in PSL(n, \mathbb{Z})$.
\end{proof}

As a direct corollary, we have the following result.

\begin{cor}
$L(SL(3,\mathbb{Z}))$ contains an index 4 von Neumann subalgebra $\mathcal{M}$ such that $\mathcal{M}$ does not contain $L(\Lambda)$ for any finite index subgroup $\Lambda$ in $SL(3,\mathbb{Z})$.
\end{cor}
\begin{proof}
By Lemma \ref{lem: PSL(n, Z) contains Z*Z_2 as a subgroup}, we deduce $\Gamma:=SL(3,\mathbb{Z})$ contains some subgroup $\langle s,t\rangle$ isomorphic to $\mathbb{Z}*\frac{\mathbb{Z}}{2\mathbb{Z}}$. We may apply Lemma \ref{lem: finite index subalgs not absorbing powers of group elements} to construct an index 4 von Neumann subalgebra $\mathcal{M}$. According to this lemma, we know $s^k\not\in \mathcal{M}$ for any $k\geq 1$. Clearly, this implies $\mathcal{M}$ does not contain $L(\Lambda)$ for any finite index subgroup $\Lambda$ in $SL(3,\mathbb{Z})$.
\end{proof}
We now give some examples of confined subalgebras which do not come from subgroups. 
\begin{example}[Confined Subalgebras not coming from Subgroups]
Let $\Gamma=\mathbb{F}_n$, the free group on $n$-generators where $n\ge 2$. Let $\{a_1,a_2,\ldots,a_n\}$ be one of the generating sets. Consider the automorphism $\sigma: \Gamma\to\Gamma$ defined by $\sigma(a_i)=a_{i+1}$ for each $i=1,2,\ldots,n$ where $a_{n+1}=a_1$. This induces an automorphism on $L(\Gamma)$, which we denote by $\sigma$ again. It is well-known that $L(\Gamma)^{\sigma}$ has finite index in $L(\Gamma)$. Therefore, using Lemma~\ref{lem: confiness passes to finite index subalgebras}, we see that $L(\Gamma)^{\sigma}$ is confined. However, it does not contain any group elements since no elements are fixed by $\sigma$.
\end{example}
\begin{remark}
It is easy to see that confineness is not a closed property. Consider, for example, a group $\Gamma$ admitting a decreasing sequence of non-trivial normal subgroups $\{N_k\}_k$ such that $\cap_kN_k=\{e\}$. It then follows that $N_k\to \{e\}$ in the Chabauty topology, hence, $L(N_k)\to\mathbb{C}$ in the EM-topology (by \cite[Proposition~4.1]{amrutam2025amenable}. However, each $L(N_k)$ is confined since $N_k$'s are all normal. But, $\mathbb{C}$ is not-confined.  
\end{remark}
We conclude this section by providing concrete, explicit examples of non-confined amenable subalgebras within free group factors. While the proof of our main general result (Theorem~\ref{thm:amenable-not-confined-direct}) will rely on the abstract topological dynamics of the Furstenberg boundary, the combinatorial geometry of free groups allows for a far more direct approach. By exploiting the ability to separate finite subsets using group elements explicitly, we can manually verify the unconfinedness of both the generator MASA and the radial MASA. 

In fact, it is easy to see that if $H$ is a finitely generated infinite index subgroup in free groups $F_n$ ($n\geq 2$), then it is not confined. This relies on the fact that finitely generated subgroups in free groups $F_n$ satisfy the \textit{Hall property}, which says that for any finitely generated subgroup $H$ in $F_n$, there exists a subgroup $K$ in $F_n$ such that $H$ is free from $K$ in the sense that the subgroup generated by $H$ and $K$ is canonically isomorphic to $H*K$; moreover, $[F_n: H*K]<\infty$ \cite{hall1949}.
\begin{lem}
Let $2\leq n\leq \infty$. Let $G=F_n$ be the non-abelian free group on $n$ generators $\{s_1,\ldots, s_n\}$. Then given any finite subsets $F\subseteq G\setminus \{e\}, F'\subseteq G$, there exists some $s\in G$  such that $sFs^{-1}\cap F'\langle s_1\rangle F'=\emptyset$.
\end{lem}
\begin{proof}
We need to find $s$ such that $F'^{-1}sFs^{-1}F'^{-1}\cap \langle s_1\rangle=\emptyset$. 

First, note that there exists some $t\in G$ such that $tFt^{-1}\cap \langle s_2\rangle=\emptyset$. Indeed, we may take  $t=s_1^k$ for large enough $k\geq 1$.

Next, we may set  $s=s_2^{\ell}t$ for large enough power $\ell\geq 1$. Indeed, for any $g, h\in F', f\in F$, $g^{-1}sfs^{-1}h^{-1}=g^{-1}s_2^{\ell}tft^{-1}s_2^{-\ell}h^{-1}$. By the choice of $t$, we know that $tft^{-1}\not\in \langle s_2\rangle$. Hence for large enough $\ell$, the letters $s_2$ and $s_2^{-1}$ could not be cancelled out in $g^{-1}s_2^{\ell}tft^{-1}s_2^{-\ell}h^{-1}$. Hence, it does not belong to $\langle s_1\rangle$. Since $F', F$ are both finite sets, we can make a uniform choice of $\ell$ satisfying the above property.
\end{proof}

\begin{example}
Let $2\leq n\leq \infty$. Let $F_n=\langle s_1,\ldots, s_n\rangle$ be the non-abelian free groups on the $n$-generators $\{s_1,\ldots, s_n\}$. Let $A=L(\langle s_1\rangle)$ be the generator MASA in $L(F_n)$. Let $u$ be any unitary in $L(F_n)$. Then $uAu^*$ is  not confined in $L(F_n)$. 
\begin{proof}
By the previous lemma, there exists a sequence of elements $\{t_k\}\subset G$ such that for any given finite set $F\subseteq G\setminus \{e\}, F'\subset G$, there exists some $K\geq 1$ such that $t_k^{-1}Ft_k\cap F'\langle s_1\rangle F'=\emptyset$ for all $k\geq K$.

We need to check that $t_kuAu^*t_k^{-1}\to \mathbb{C}$ under the EM-topology. Thus, we need to show that for any $e\neq g\in F_n$,  $\left\|E_A(u^*t_k^{-1}gt_ku)\right\|_2\to 0$.

Note that $\{s_1^n: n\in\mathbb{Z}\}$ is an orthogonal normal basis for $A$ w.r.t. the $\left\|\cdot\right\|_2$-norm. Thus, $E_A(x)=\sum_{i\in\mathbb{Z}}\tau(xs_1^{-i})s_1^i$ for any $x\in L(F_n)$.
For any given $\epsilon>0$, there exists some finite set $F'\subset G$ and some $u_{F'}\in \mathbb{C}G$ with $\text{supp}(u_{F'})\subseteq F'$  such that $\left\|u-u_{F'}\right\|_2<\epsilon$ and $\left\|u_{F'}\right\|\leq 1$.
Note that for any $g\in G$, we have 
\begin{align*}
\left\|E_A(u^*t_k^{-1}gt_ku)\right\|_2
&\leq \left\|E_A(u_{F'}^*t_k^{-1}gt_ku_{F'})\right\|_2+\left\|E_A(u^*t_k^{-1}gt_ku-u_{F'}^*t_k^{-1}gt_ku_{F'})\right\|_2\\
&\leq \left\|E_A(u_{F'}^*t_k^{-1}gt_ku_{F'})\right\|_2+\left\|u^*t_k^{-1}gt_ku-u_{F'}^*t_k^{-1}gt_ku_{F'}\right\|_2\\
&\leq \left\|E_A(u_{F'}^*t_k^{-1}gt_ku_{F'})\right\|_2+2\left\|u-u_{F'}\right\|_2\\
&\leq \left\|E_A(u_{F'}^*t_k^{-1}gt_ku_{F'})\right\|_2+2\epsilon.
\end{align*}

Hence, we deduce that
\begin{align*}
    \left\|E_A(u^*t_k^{-1}gt_ku)\right\|_2^2&\leq  
    2\left\|E_A(u_{F'}^*t_k^{-1}gt_ku_{F'})\right\|_2^2+2(2\epsilon)^2\\
    &=2\sum_{i\in\mathbb{Z}}|\tau(u_{F'}^*t_k^{-1}gt_ku_{F'}s_1^{-i})|^2+8\epsilon^2\\
    &=2\sum_{i\in\mathbb{Z}}|\tau(t_k^{-1}gt_ku_{F'}s_1^{-i}u_{F'}^*)|^2+8\epsilon^2\\
    &=8\epsilon^2 (\text{as $k\to\infty$}).
\end{align*}
Note that the last equality holds by taking $k\geq K$ where $K$ satisfies that $t_K^{-1}\{g\}t_K\cap (F'\cup F'^{-1})\langle s_1\rangle (F'\cup F'^{-1})=\emptyset$.
\end{proof}
\end{example}
\begin{example}
Let $B=\{\sum_{i=1}^ns_i+s_i^{-1}\}''$ be the radial MASA in $L(F_n)$, where $n\geq 2$ and $\{s_1,\ldots, s_n\}$ is a set of free generators in $F_n$. Then $uBu^*$ is not confined for any unitary operator $u\in L(F_n)$.
\begin{proof}
Recall that $L^2(B)$
has an orthogonal basis $\{w_i/\|w_i\|_2\}_{i=0}^{\infty}$, where $w_i\in L(F_n)$ is the sum of all words of length $i$ in $F_n$. Note that $w_i$ is self-adjoint for all $i$.

Therefore, for any $y\in L(F_n)$, we have the following holds under $\|\cdot\|_2$-norm.
\begin{align*}
E_B(y)=\sum_{i=0}^{\infty}\tau(\frac{yw_i}{\|w_i\|_2})\frac{w_i}{\|w_i\|_2}.
\end{align*}
Since $F_n$ is i.c.c., we may find a sequence $s_m\in F_n$ such that $s_m^{-1}Fs_m\cap F=\emptyset$ for all finite subset $F\subset F_n\setminus \{e\}$ as $m\rightarrow \infty$. We claim that $s_muBu^*s_m^{-1}\rightarrow \mathbb{C}$ under EM-topology.

To show this, it suffices to check that for any $e\neq g\in F_n$, we have $E_B(u^*s_m^{-1}gs_mu)\rightarrow 0$ under $\|\cdot\|_2$-norm as $m\rightarrow \infty$.

First, we observe that for any words $t_k\in F_n$ such that $t_k\rightarrow \infty$, we have $\|E_B(t_k)\|_2\rightarrow 0$. 

Indeed, since $E_B(t_k)=\sum_{i=0}^{\infty}\tau(\frac{t_kw_i}{\|w_i\|_2})\frac{w_i}{\|w_i\|_2}$, we deduce that 
\begin{align*}
\|E_B(t_k)\|_2^2=\sum_{i=0}^{\infty}\frac{\left|\tau(t_kw_i)\right|^2}{\|w_i\|_2^2} 
=\frac{1}{\|w_{\ell(t_k)}\|^2}=\frac{1}{2n(2n-1)^{\ell(t_k)}}\rightarrow 0,
\end{align*}
where $\ell(t_k)$ denotes the word length of $t_k$ w.r.t. the standard  generating set $\{s_i^{\pm}: 1\leq i\leq n\}$.

Next, given any $\epsilon>0$, we may pick $u_{F}\subset \mathbb{C}G$ with $\text{supp}(u_F)\subseteq F$, $\left\|u-u_F\right\|_2<\epsilon$ and $\left\|u_F\right\|\leq 1$, where $F\subset G$ is a finite subset. 
Set $r=\sup_{g\in F}\ell(g)$, where $\ell(g)$ denotes the word length of $G$. Write $B_r(e)$ be the ball in $F_n$ with radius $\leq r$, note that $\sharp B_r(e)\leq (2n)^r$.
Note that 
\begin{align*}
    \left\|E_B(u^*s_m^{-1}gs_mu) \right\|\leq \left\| E_B(u_F^*s_m^{-1}gs_mu_F)\right\|_2+\left\| E_B(u^*s_m^{-1}gs_mu)- E_B(u_F^*s_m^{-1}gs_mu_F)\right\|\\
    \leq \left\| E_B(u_F^*s_m^{-1}gs_mu_F)\right\|_2+2\left\|u-u_F\right\|_2\leq  \left\| E_B(u_F^*s_m^{-1}gs_mu_F)\right\|_2+2\epsilon.
\end{align*}
Then notice that $E_B(u_F^*s_m^{-1}gs_mu_F)$ can be written as a linear combination of finitely many (in fact at most $[\sharp B(r)]^2$-many) terms in the form $E_B(t_1^{-1}s_m^{-1}gs_mt_2)$, where $t_1,t_2\in \supp(u_F)$. Then by choosing $m$ large enough and applying what we have proved above to $t_k:=t_1^{-1}s_m^{-1}gs_mt_2$ for all $t_1,t_2\in\supp(u_F)$, we may assume that $\left\| E_B(u_F^*s_m^{-1}gs_mu_F)\right\|_2\leq \epsilon$. This is enough to finish the proof.
\end{proof}
\end{example}
\begin{example}
Let $G=F_2\rtimes_{\sigma} \frac{\mathbb{Z}}{2\mathbb{Z}}$, where $\frac{\mathbb{Z}}{2\mathbb{Z}}=\langle s\rangle$ and $\sigma$ is the automorphism defined by $\sigma_s(a)=b$ and $\sigma_s(b)=a$ on the free group $F_2=\langle a,b\rangle$. Let $A=\{a+b+a^{-1}+b^{-1}\}'' \subset L(F_2)$ be the radial MASA. Define $B := A\rtimes_{\sigma}\frac{\mathbb{Z}}{2\mathbb{Z}} = A \bar{\otimes} L(\frac{\mathbb{Z}}{2\mathbb{Z}})$. Then $B$ is not confined in $L(G)$.
\begin{proof}
We will construct a sequence $\{h_n\} \subset F_2 \le G$ such that $h_n B h_n^{-1} \to \mathbb{C}$ in the EM-topology. By Proposition~\ref{prop:easyobservations}, it suffices to show that for every $g \in G \setminus \{e\}$, the conditional expectation satisfies
\[
\lim_{n \to \infty} \left\|E_B(h_n^{-1} g h_n)\right\|_2 = 0.
\]
Let $h_n = a^n b^n \in F_2$. Any element $g \in G$ can be written uniquely as either $g = x$ or $g = xs$ for some $x \in F_2$. We analyze the conditional expectation in these two cases. 

For the radial MASA $A$, we already know that for any sequence of words $\{w_n\} \subset F_2$, if the word length $|w_n| \to \infty$, then $\|E_A(w_n)\|_2 \to 0$.

\noindent \textbf{Case 1: $g = x \in F_2 \setminus \{e\}$.} \\
Conjugating $x$ by $h_n$, we obtain:
\[
h_n^{-1} x h_n = b^{-n} a^{-n} x a^n b^n.
\]
Since $x \neq e$, let us look at the reduced word length of this element. For $n > |x|$, the $a^{-n}$ and $a^n$ terms cannot be completely cancelled by $x$ unless  $x\in \langle a\rangle$. Even in the worst-case scenario where $x = a^k$ for some $k$, the conjugated element becomes $b^{-n} a^k b^n$, which has no further cancellation and a word length of $2n + |k|$. Thus, for any fixed $x \neq e$, the word length $|h_n^{-1} x h_n| \to \infty$ as $n \to \infty$. Consequently,
\[
\left\|E_B(h_n^{-1} x h_n)\right\|_2 = \left\|E_A(b^{-n} a^{-n} x a^n b^n)\right\|_2 \xrightarrow{n \to \infty} 0.
\]

\noindent \textbf{Case 2: $g = xs$ for some $x \in F_2$.} \\
Using the relation $s h_n = \sigma_s(h_n) s$, we conjugate $g$ by $h_n$:
\begin{align*}
h_n^{-1} (xs) h_n &= h_n^{-1} x (s h_n s) s \\
&= h_n^{-1} x \sigma_s(h_n) s.
\end{align*}
Since $h_n = a^n b^n$ and $\sigma_s$ swaps the generators, we have $\sigma_s(h_n) = b^n a^n$. Substituting this yields:
\[
h_n^{-1} (xs) h_n = (b^{-n} a^{-n}) x (b^n a^n) s.
\]
Applying the conditional expectation $E_B$, we get:
\[
\left\|E_B(h_n^{-1} x s h_n)\right\|_2 = \left\|E_A(b^{-n} a^{-n} x b^n a^n) s \right\|_2 = \left\|E_A(b^{-n} a^{-n} x b^n a^n)\right\|_2.
\]
We now examine the word length of $y_n = b^{-n} a^{-n} x b^n a^n$. For any fixed $x$, choose $n > |x|$. The subword $a^{-n}$ on the left and the subword $b^n$ on the right are distinct generators and cannot annihilate each other through the finite word $x$. The maximum possible cancellation is $|x|$, meaning the reduced word length satisfies $|y_n| \ge 4n - |x|$. Therefore, $|y_n| \to \infty$ as $n \to \infty$, which implies:
\[
\left\|E_A(b^{-n} a^{-n} x b^n a^n)\right\|_2 \xrightarrow{n \to \infty} 0.
\]
Since the conditional expectation vanishes asymptotically for all $g \neq e$, it follows that $h_n B h_n^{-1} \to \mathbb{C}$ in the Effros-Mar\'{e}chal topology. Thus, $B$ is not confined.
\end{proof}
\end{example}

\begin{remark}
Of course, many non-amenable subalgebras are not confined. $F_2$ contains $F_{\infty}$ as a subgroup. Hence, pick $F_n \le F_{\infty}$ and an element $s$ which is free from $F_n$. It then follows that $\lambda(s)^mF_n\lambda(s)^{-m}\to \mathbb{C}$ in the Effros-Mar\'{e}chal topology.  
\end{remark}

\section{A positive definite function approach to confinement}
\label{sec:pdfandURA}
As established in Section~\ref{sec:preliminaries}, the non-compactness of the Effros-Mar\'{e}chal topology presents a formidable analytical hurdle for studying the global dynamics of $\text{SubAlg}(L(\Gamma))$. To bypass this obstruction, we transport our dynamical system to a more tractable topological space via positive definite functions.

Recall that $\mathrm{PD}_1(\Gamma)$ is the space of normalized positive definite functions on $\Gamma$, equipped with the pointwise topology. Since $\mathrm{PD}_1(\Gamma)\subset [-1,1]^{\Gamma}$, it is a compact Hausdorff space. The group $\Gamma$ acts on $\mathrm{PD}_1(\Gamma)$ continuously by conjugation:
\[
(s\cdot\phi)(g) := \phi(s^{-1}gs), \quad s,g\in\Gamma,\ \phi\in\mathrm{PD}_1(\Gamma).
\]
\begin{definition}
Let $\mathcal{M} \le L(\Gamma)$ be a von Neumann subalgebra. Define
$\phi_{\mathcal{M}} : \Gamma \to \mathbb{C}$ by
\[
\phi_{\mathcal{M}}(g) := \tau_0\!\left(E_{\mathcal{M}}(\lambda_g)\,\lambda_g^*\right)
= \left\|E_{\mathcal{M}}(\lambda_g)\right\|_2^2,
\]
where the second equality holds since $E_{\mathcal{M}}$ is a self-adjoint
projection on $L^2(L(\Gamma))$.
\end{definition}

\begin{remark}
Note the following basic properties of $\phi_{\mathcal{M}}$.
\begin{enumerate}
    \item $\phi_{\mathcal{M}}(e) = 1$ and $0 \le \phi_{\mathcal{M}}(g) \le 1$
    for all $g \in \Gamma$.
    \item For the extreme cases: $\phi_{\mathbb{C}} = \delta_e$ and
    $\phi_{L(\Gamma)} \equiv 1$.
    \item For a subgroup $H \le \Gamma$ one has
    $\phi_{L(H)} = \mathbf{1}_H$, the indicator function of $H$.
    \item The assignment $\mathcal{M} \mapsto \phi_{\mathcal{M}}$ is
    $\Gamma$-equivariant: for every $s \in \Gamma$,
    \[
        (s \cdot \phi_{\mathcal{M}})(g)
        := \phi_{\mathcal{M}}(s^{-1}gs)
        = \phi_{s\mathcal{M}s^{-1}}(g).
    \]
    \item $\phi_{\mathcal{M}}$ is a positive definite function on $\Gamma$.
\end{enumerate}
\end{remark}
Crucially, this association perfectly bridges our topological spaces. The following proposition proves that the assignment $\mathcal{M} \mapsto \phi_{\mathcal{M}}$ acts as a continuous, $\Gamma$-equivariant map from the non-compact Effros-Mar\'{e}chal space into the pointwise compact space of normalized positive definite functions.
\begin{prop}
\label{prop:continuity-equivariance-of-phi}
The map $\Phi: \mathrm{SubAlg}(L(\Gamma)) \to \mathrm{PD}_1(\Gamma)$ defined by $\mathcal{M} \mapsto \phi_{\mathcal{M}}$ is a continuous, $\Gamma$-equivariant map, where $\mathrm{SubAlg}(L(\Gamma))$ is equipped with the Effros-Mar\'{e}chal topology and $\mathrm{PD}_1(\Gamma)$ is equipped with the pointwise topology.
\end{prop}
\begin{proof}
We first show continuity. Let $\{\mathcal{M}_n\}_{n=1}^{\infty}$ be a sequence of von Neumann subalgebras converging to $\mathcal{M}$ in the Effros-Mar\'{e}chal topology. By Proposition~\ref{conditionalexpectationconvergence}, the corresponding canonical conditional expectations converge in the strong-$*$ topology:
\[
    E_{\mathcal{M}_n}(x) \xrightarrow[]{\text{so*}}  E_{\mathcal{M}}(x) \quad \text{for all } x \in L(\Gamma).
\]
In a finite von Neumann algebra, strong-$*$ convergence implies convergence in the $\|\cdot\|_2$-norm. Therefore, for any fixed $g \in \Gamma$, we have
\[
    \lim_{n \to \infty} \left\| E_{\mathcal{M}_n}(\lambda_g) \right\|_2 = \left\| E_{\mathcal{M}}(\lambda_g) \right\|_2.
\]
Squaring both sides yields $\lim_{n \to \infty} \phi_{\mathcal{M}_n}(g) = \phi_{\mathcal{M}}(g)$. Since this holds for every $g \in \Gamma$, the map $\Phi$ is continuous.

Next, we verify $\Gamma$-equivariance. The group $\Gamma$ acts on $\mathrm{SubAlg}(L(\Gamma))$ by conjugation, $s \cdot \mathcal{M} = \lambda_s \mathcal{M} \lambda_s^*$, and on $\mathrm{PD}_1(\Gamma)$ by $(s \cdot \phi)(g) = \phi(s^{-1}gs)$. Using Lemma~\ref{canonicalconditionalexpectation} and the tracial property of $\tau_0$, we see that
\begin{align*}
    \Phi(s \cdot \mathcal{M})(g) &= \phi_{\lambda_s \mathcal{M} \lambda_s^*}(g) \\
    &= \tau_0\!\left( E_{\lambda_s \mathcal{M} \lambda_s^*}(\lambda_g) \lambda_g^* \right) \\
    &= \tau_0\!\left( \lambda_s E_{\mathcal{M}}(\lambda_s^* \lambda_g \lambda_s) \lambda_s^* \lambda_g^* \right) \\
    &= \tau_0\!\left( E_{\mathcal{M}}(\lambda_{s^{-1}gs}) \lambda_s^* \lambda_g^* \lambda_s \right) \\
    &= \tau_0\!\left( E_{\mathcal{M}}(\lambda_{s^{-1}gs}) \lambda_{s^{-1}gs}^* \right) \\
    &= \phi_{\mathcal{M}}(s^{-1}gs) \\
    &= (s \cdot \Phi(\mathcal{M}))(g).
\end{align*}
Thus, $\Phi$ is $\Gamma$-equivariant.
\end{proof}
With this continuous map established, we can entirely reformulate the geometric concept of confinement through the lens of positive definite functions.

The key dynamical observation is that non-confinement in the Effros-Mar\'{e}chal topology translates precisely to the pointwise decay of our associated positive definite function to $\delta_e$.
\begin{prop}
\label{prop:phiM-orbit}
Let $\mathcal{M} \le L(\Gamma)$ be a von Neumann subalgebra and let
$\{s_i\} \subset \Gamma$ be a sequence. The following are equivalent.
\begin{enumerate}
    \item $s_i \mathcal{M} s_i^{-1} \to \mathbb{C}$ in
    $\mathrm{EM}$-topology.
    \item $s_i \cdot \phi_{\mathcal{M}} \to \delta_e$ pointwise on
    $\Gamma$.
\end{enumerate}
\begin{proof}
$(1) \Rightarrow (2)$: Fix $g \ne e$. By
Proposition~\ref{prop:easyobservations}, condition (1) gives
$\left\|E_{\mathcal{M}}(\lambda_{s_i^{-1}gs_i})\right\|_2 \to 0$.
Hence, using the $\Gamma$-equivariance, we get that
\[
(s_i \cdot \phi_{\mathcal{M}})(g)
= \phi_{\mathcal{M}}(s_i^{-1}gs_i)
= \left\|E_{\mathcal{M}}(\lambda_{s_i^{-1}gs_i})\right\|_2^2
\xrightarrow{i \to \infty} 0
= \delta_e(g).
\]
At $g = e$ we always have $(s_i \cdot \phi_{\mathcal{M}})(e) = 1 =
\delta_e(e)$, so pointwise convergence holds on all of $\Gamma$.

$(2) \Rightarrow (1)$: If $s_i \cdot \phi_{\mathcal{M}} \to \delta_e$
pointwise, then for every $g \ne e$,
\[
\left\|E_{\mathcal{M}}(\lambda_{s_i^{-1}gs_i})\right\|_2^2
= \phi_{\mathcal{M}}(s_i^{-1}gs_i)
= (s_i \cdot \phi_{\mathcal{M}})(g) \to 0,
\]
which by Proposition~\ref{prop:easyobservations}(4) is equivalent to
$s_i \mathcal{M} s_i^{-1} \to \mathbb{C}$ in EM-topology.
\end{proof}
\end{prop}

\begin{cor}
\label{cor:confinement-pdfn}
Let $\mathcal{M} \le L(\Gamma)$ be a von Neumann subalgebra. Then:
\[
\mathcal{M}\ \text{is confined}
\quad \Longleftrightarrow \quad
\delta_e \notin
\overline{\left\{s \cdot \phi_{\mathcal{M}} : s \in \Gamma\right\}}^{\,\mathrm{ptwise}}.
\]
That is, $\mathcal{M}$ is confined if and only if $\delta_e$ is not in
the pointwise orbit closure of $\phi_{\mathcal{M}}$ in
$\mathrm{PD}_1(\Gamma)$.
\end{cor}
In a $C^*$-simple group, amenability intrinsically forces non-confinement. This is the main result of this section and completes the proof of Theorem~\ref{thm:mainconfined}. 

We first recall the notion of boundary actions. Let $\Gamma$ be a discrete countable group and $X$ be a $\Gamma$-space, i.e., $X$ is a compact Hausdorff space and $\Gamma\curvearrowright X$ by homeomorphisms. The action $\Gamma \curvearrowright X$ is called a boundary action if
\[
\{\delta_x: x\in X\}\subset\overline{\Gamma\nu}^{\text{weak}^*}
\]
for every $\nu \in \text{Prob}(X)$. The Furstenberg boundary of $\Gamma$, denoted $\partial_F\Gamma$, is a $\Gamma$-boundary which is universal in the sense that every other $\Gamma$-boundary $Y$ is a $\Gamma$-equivariant continuous image of $\partial_F\Gamma$.  Let $X$ be a minimal $\Gamma$-space (such as the Furstenberg boundary $\partial_F\Gamma$). We can view the continuous functions $C(X)$ as multiplication operators on $\mathbb{B}(\ell^2(\Gamma))$. 

Fix a base point $x_0\in X$. For any $f\in C(X)$, the map $M(f):\ell^2(\Gamma)\to\ell^2(\Gamma)$ defined by 
\[
M(f)(\delta_t)=f(t.x_0)\delta_t
\]
is linear and bounded. Since the action $\Gamma\curvearrowright X$ is minimal, we see that $\|M(f)\|=\|f\|_{\infty}$ for every $f\in C(X)$. Therefore, we can identify $C(X)$ with its image $M(C(X))$ inside $\mathbb{B}(\ell^2(\Gamma))$. Note that while this embedding is faithful, it is not canonical as it depends on the choice of the base point $x_0$.
\begin{theorem}
\label{thm:amenable-not-confined-direct}
Let $\Gamma$ be a $C^*$-simple group and $\mathcal{M} \le L(\Gamma)$ an amenable von Neumann subalgebra. Then there exists a sequence $\{s_i\} \subset \Gamma$ such that $s_i \cdot \phi_{\mathcal{M}} \to \delta_e$ pointwise. In particular, $\mathcal{M}$ is not confined.
\begin{proof}
Because $\mathcal{M}$ is amenable, using \cite[Proposition~2.1]{amrutam2025amenable},  we find a hypertrace $\varphi \in \mathrm{Hype}_{\tau_0}(\mathcal{M})$. Using the embedding $f \mapsto M(f)$ of $C(\partial_F\Gamma)$ into $\mathbb{B}(\ell^2(\Gamma))$ defined above, we restrict the state $\varphi$ to $C(\partial_F\Gamma)$; by the Riesz representation theorem, this yields a probability measure $\mu$ on $\partial_F\Gamma$.

Since $\Gamma\curvearrowright \partial_F \Gamma$ is a boundary action, we can find a net $\{s_i\} \subset \Gamma$ and a point $x \in \partial_F \Gamma$ such that $s_i \mu \xrightarrow{\text{weak}^*} \delta_x$.
Define the sequence of conjugated states $\varphi_i = s_i \cdot \varphi = \varphi(\lambda_{s_i}^* \cdot \lambda_{s_i})$. Observe that $\varphi_i\in\text{Hype}_{\tau_0}(\lambda(s_i)\mathcal{M}\lambda(s_i)^*):=\mathcal{M}_i$. By passing to a subnet if necessary, $\varphi_i$ converges in the weak*-topology to a state $\psi$ on $B(\ell^2(\Gamma))$. Observe that $\psi|_{C(\partial_F \Gamma)} = \delta_x$.

We now evaluate the pointwise action of $s_i$ on the positive definite function $\phi_{\mathcal{M}}$. By definition and the hypertrace property of $\varphi$, we see that
\begin{align*}
(s_i \cdot \phi_{\mathcal{M}})(g) &= \phi_{\mathcal{M}}(s_i^{-1} g s_i) \\
&= \tau_0\!\left( E_{\mathcal{M}}(\lambda_{s_i^{-1} g s_i}) \lambda_{s_i^{-1} g^{-1} s_i} \right) \\
&= \varphi\!\left( E_{\mathcal{M}}(\lambda_{s_i}^* \lambda_g \lambda_{s_i}) \lambda_{s_i}^* \lambda_g^* \lambda_{s_i} \right).
\end{align*}
Using the fact that $\varphi_i(\cdot) = \varphi(\lambda_{s_i}^* \cdot \lambda_{s_i})$, we can rewrite this as:
\begin{align*}
&(s_i \cdot \phi_{\mathcal{M}})(g)=\varphi\!\left( E_{\mathcal{M}}(\lambda_{s_i}^* \lambda_g \lambda_{s_i}) \lambda_{s_i}^* \lambda_g^* \lambda_{s_i} \right)\\&=\varphi\!\left(\lambda(s_i)^*\left(\lambda(s_i) E_{\mathcal{M}}(\lambda_{s_i}^* \lambda_g) \lambda_{s_i}) \lambda_{s_i}^* \lambda_g^*\right) \lambda(s_i) \right)\\& = \varphi_i\!\left( \lambda_{s_i} E_{\mathcal{M}}(\lambda_{s_i}^* \lambda_g \lambda_{s_i}) \lambda_{s_i}^* \lambda_g^* \right).
\end{align*}
Notice that the term $\lambda_{s_i} E_{\mathcal{M}}(\lambda_{s_i}^* \lambda_g \lambda_{s_i}) \lambda_{s_i}^*$ is exactly the conditional expectation onto the conjugated algebra $\mathcal{M}_i := s_i \mathcal{M} s_i^{-1}$. Let us denote $m_i := E_{\mathcal{M}_i}(\lambda_g)$. Thus, it follows that
\[
(s_i \cdot \phi_{\mathcal{M}})(g) = \varphi_i(m_i \lambda_g^*).
\]
We must show this sequence goes to $0$ for any $g \neq e$. Since $\Gamma$ is $C^*$-simple, the action $\Gamma\curvearrowright\partial_F\Gamma$ is free (cf.~\cite{breuillard2017c}), so $g^{-1}x \neq x$. We can then find a continuous function $f \in C(\partial_F \Gamma)$ such that $0 \le f \le 1$, $f(x) = 1$, and $f(gx) = 0$. We split the evaluation using $f$ as a cutoff function and write
\[
\varphi_i(m_i \lambda_g^*) = \varphi_i(m_i \lambda_g^* f) + \varphi_i(m_i \lambda_g^* (1-f)).
\]
We deal with each term individually. Commuting the $\Gamma$-action on $f$, we have $\varphi_i(m_i \lambda_g^* f) = \varphi_i(m_i (g^{-1}\cdot f) \lambda_g^*)$. 
By the Cauchy-Schwarz inequality for the state $\varphi_i$:
\[
\left| \varphi_i\!\left( m_i \sqrt{g^{-1}\cdot f} \sqrt{g^{-1}\cdot f} \lambda_g^* \right) \right|^2 \le \varphi_i\!\left( m_i (g^{-1}\cdot f) m_i^* \right) \varphi_i\!\left( \lambda_g (g^{-1}\cdot f) \lambda_g^* \right).
\]
Because $\varphi_i \in\text{Hype}_{\tau_0}( \mathcal{M}_i)$, and $m_i \in \mathcal{M}_i$, it centralizes $m_i$. Therefore, $\varphi_i( m_i (g^{-1}\cdot f) m_i^* ) = \varphi_i( (g^{-1}\cdot f) m_i^* m_i )$. Since $E_{\mathcal{M}_i}$ is a conditional expectation, we have by the Kadison-Schwarz inequality,
\[
m_i^* m_i = E_{\mathcal{M}_i}(\lambda_g)^* E_{\mathcal{M}_i}(\lambda_g) \le E_{\mathcal{M}_i}(\lambda_g^* \lambda_g) = E_{\mathcal{M}_i}(1) = 1.
\]
Thus, $1 - m_i^* m_i \ge 0$. We consider the difference:
\begin{align*}
\varphi_i(g^{-1}\cdot f) - \varphi_i\!\left((g^{-1}\cdot f) m_i^* m_i\right) 
&= \varphi_i\!\left( (g^{-1}\cdot f) (1 - m_i^* m_i) \right) \\
&= \varphi_i\!\left( (g^{-1}\cdot f) (1 - m_i^* m_i)^{1/2} (1 - m_i^* m_i)^{1/2} \right).
\end{align*}
Using the hypertrace property of $\varphi_i$ to cyclically permute $(1 - m_i^* m_i)^{1/2} \in \mathcal{M}_i$, we obtain
\begin{align*}
&\varphi_i\!\left( (1 - m_i^* m_i)^{1/2} (g^{-1}\cdot f) (1 - m_i^* m_i)^{1/2} \right) 
\\&= \varphi_i\!\left( (1 - m_i^* m_i)^{1/2} (g^{-1}\cdot f)^{1/2} (g^{-1}\cdot f)^{1/2} (1 - m_i^* m_i)^{1/2} \right) \\
&= \varphi_i\!\left( X X^* \right) \ge 0,
\end{align*}
where $X = (1 - m_i^* m_i)^{1/2} (g^{-1}\cdot f)^{1/2}$. 
This implies that $\varphi_i\!\left((g^{-1}\cdot f) m_i^* m_i\right) \le \varphi_i(g^{-1}\cdot f)$.

As $i \to \infty$, $\varphi_i(g^{-1}\cdot f) \to \psi(g^{-1}\cdot f) = f(gx) = 0$. The second factor $\varphi_i(\lambda_g (g^{-1}\cdot f) \lambda_g^*) = \varphi_i(f)$ is bounded uniformly by $1$. Hence, it follows that
\[\lim_{i\to\infty}|\varphi_i( m_i \lambda_g^*f )|=0.\]
Applying Cauchy-Schwarz to the second term, we see that
\[
\left| \varphi_i\!\left( m_i \lambda_g^* \sqrt{1-f} \sqrt{1-f} \right) \right|^2 \le \varphi_i\!\left( m_i \lambda_g^* (1-f) \lambda_g m_i^* \right) \varphi_i(1-f).
\]
As $i \to \infty$, $\varphi_i(1-f) \to \psi(1-f) = 1 - f(x) = 0$. The first factor is bounded by the operator norm $\|m_i \lambda_g^* (1-f) \lambda_g m_i^*\|_\infty \le 1$. Hence, the second term goes to $0$ as well.

Since both terms vanish, we conclude that for every $g \neq e$, 
\[
\lim_{i \to \infty} (s_i \cdot \phi_{\mathcal{M}})(g) = 0.
\]
Since $(s_i \cdot \phi_{\mathcal{M}})(e) = 1$ always, we have $s_i \cdot \phi_{\mathcal{M}} \to \delta_e$ pointwise. By Proposition~\ref{prop:phiM-orbit}, $\mathcal{M}$ is not confined.
\end{proof}
\end{theorem}

\subsection{Uniformly Recurrent States}
\label{subsec:URA}
To bypass the non-compactness of the Effros-Mar\'{e}chal topology on the space of von Neumann subalgebras $\mathrm{SubAlg}(L(\Gamma))$, we transport the dynamics to the state space via the map $\mathcal{M} \mapsto \phi_{\mathcal{M}}$.

We define the space of \emph{subalgebra states} $\mathcal{X}_\Gamma \subset \mathrm{PD}_1(\Gamma)$ as the pointwise closure of the positive definite functions arising from von Neumann subalgebras:
\[
    \mathcal{X}_\Gamma := \overline{\{\phi_{\mathcal{M}} : \mathcal{M} \le L(\Gamma)\}}^{\,\mathrm{ptwise}}.
\]
Because $\mathcal{X}_\Gamma$ is a closed subset of a compact space, it is itself a compact $\Gamma$-space. This serves as our operator-algebraic replacement for the Chabauty space of subgroups.

\begin{definition}[Uniformly Recurrent State]
\label{def:URA}
A \emph{Uniformly Recurrent State} is a non-empty, closed, minimal $\Gamma$-invariant subset of $\mathcal{X}_\Gamma$.
\end{definition}

\begin{remark}
Since $\mathcal{X}_\Gamma$ is compact, Zorn's Lemma guarantees that every non-empty closed $\Gamma$-invariant subset of $\mathcal{X}_\Gamma$ contains at least one uniformly recurrent state. In particular, for any von Neumann subalgebra $\mathcal{M}\leq L(\Gamma)$, the pointwise orbit closure $\overline{\{s\cdot\phi_{\mathcal{M}}: s\in\Gamma\}}^{\,\mathrm{ptwise}}$ is guaranteed to contain at least one Uniformly Recurrent State.
\end{remark}
Furthermore, every Uniformly Recurrent Subgroup intrinsically generates a corresponding Uniformly Recurrent State.
\begin{lemma}
\label{lem:URS-is-URA}
Let $\mathcal{X}\in\mathrm{URS}(\Gamma)$ be a uniformly recurrent subgroup. Then
\[
\mathcal{H}_{\mathcal{X}} := \{\phi_{L(H)} : H\in\mathcal{X}\} = \{\mathbf{1}_H : H\in\mathcal{X}\}
\]
is a uniformly recurrent state. Moreover, $\mathcal{H}_{\mathcal{X}}$ is amenable if and only if $\mathcal{X}$ is amenable, and $\mathcal{H}_{\mathcal{X}}$ is trivial if and only if $\mathcal{X}=\{\{e\}\}$.
\end{lemma}
\begin{proof}
We first observe that $\phi_{L(H)}=\mathbf{1}_H$ for any subgroup $H\leq\Gamma$, since
\[
\phi_{L(H)}(g) = \|E_{L(H)}(\lambda_g)\|_2^2 = \begin{cases} 1 & g\in H \\ 0 & g\notin H \end{cases} = \mathbf{1}_H(g).
\]
The map $H\mapsto\mathbf{1}_H$ is therefore a $\Gamma$-equivariant bijection from $\mathrm{Sub}(\Gamma)$ to the corresponding subset of $\mathcal{SA}(L(\Gamma))$, where equivariance follows from
\[
s\cdot\mathbf{1}_H = \mathbf{1}_{sHs^{-1}} = \phi_{L(sHs^{-1})} = s\cdot\phi_{L(H)}.
\]
Moreover, this map is a homeomorphism from $\mathrm{Sub}(\Gamma)$ with the Chabauty topology onto its image in $\mathcal{SA}(L(\Gamma))$ with the pointwise topology. Indeed, a net $H_i\to H$ in the Chabauty topology if and only if $\mathbf{1}_{H_i}\to\mathbf{1}_H$ pointwise, by definition of the Chabauty topology. Consequently, $L(H_i)\to L(H)$ using \cite[Proposition~4.1]{amrutam2025amenable}.

Since $\mathcal{X}$ is a URS, it is non-empty, closed, minimal, and $\Gamma$-invariant in $\mathrm{Sub}(\Gamma)$. The homeomorphism $H\mapsto\mathbf{1}_H$ transfers all four properties to $\mathcal{H}_{\mathcal{X}}$, making it a non-empty, closed, minimal, $\Gamma$-invariant subset of $\mathcal{SA}(L(\Gamma))$, i.e.\ a Uniformly Recurrent State.

For the second statement, $\mathcal{H}_{\mathcal{X}}$ is amenable if and only if $L(H)$ is amenable for every $H\in\mathcal{X}$, which holds if and only if $H$ is an amenable group for every $H\in\mathcal{X}$, i.e.\ $\mathcal{X}$ is an amenable URS.

Finally, $\mathcal{H}_{\mathcal{X}}=\{\delta_e\}$ if and only if $\mathbf{1}_H=\delta_e$ for all $H\in\mathcal{X}$, which happens if and only if $H=\{e\}$ for all $H\in\mathcal{X}$, i.e.\ $\mathcal{X}=\{\{e\}\}$.
\end{proof}
\begin{remark}
Lemma~\ref{lem:URS-is-URA} shows that the assignment $\mathcal{X}\mapsto\mathcal{H}_{\mathcal{X}}$ embeds $\mathrm{URS}(\Gamma)$ into the collection of uniformly recurrent states as a $\Gamma$-equivariant, order-preserving injection. In particular, every non-trivial amenable URS gives rise to a non-trivial amenable uniformly recurrent state.
\end{remark}
We now complete the proof of Theorem~\ref{thm:mainconfined} from the introduction.
\begin{proof}[Proof of Theorem~\ref{thm:mainconfined}]The forward direction (C$^*$-simplicity implies no amenable confined subalgebra) is the content of Theorem~\ref{thm:amenable-not-confined-direct}. For the converse, suppose $\Gamma$ is not $C^*$-simple. Using~\cite[Theorem~4.1]{Ken20}, $\Gamma$ admits a non-trivial amenable URS $\mathcal{X} \in \mathrm{URS}(\Gamma)$. Since $\mathcal{X}$ is non-trivial, $\{e\} \notin \mathcal{X}$, so for every $H \in \mathcal{X}$ the Chabauty-closure of $\{H^g : g \in \Gamma\}$ does not contain $\{e\}$. By the homeomorphism $H \mapsto \mathbf{1}_H$ established in the proof of Lemma~\ref{lem:URS-is-URA} below, this implies $\mathbb{C} \notin \overline{\{L(H)^g : g \in \Gamma\}}^{\mathrm{EM}}$, so $L(H)$ is a confined amenable subalgebra of $L(\Gamma)$. This completes the proof of Theorem~\ref{thm:mainconfined}. \end{proof}

\subsection{Exotic Uniform Recurrent States: Subalgebras from Automorphisms}
\label{subsec:exoticURAs}
In this subsection, we show that there exist \emph{exotic} uniformly recurrent states---ones that cannot be realized as $\mathcal{H}_{\mathcal{X}}$ for any Uniformly Recurrent Subgroup $\mathcal{X}$. We construct such examples by analyzing the fixed-point von Neumann subalgebras associated with finite-order automorphisms.

\begin{definition}[Subalgebra associated with an automorphism]
 Let $\varphi\in\text{Aut}(\Gamma)$ be a non-trivial automorphism.
Let $\mathcal{N}_{\varphi}$ be a collection of elements of $L(\Gamma)$ defined by
\begin{align*}\mathcal{N}_{\varphi}&=\left\{\sum_{g\in\Gamma}c_g\lambda(g): c_g=c_{\varphi(g)}\right\}\\&=\left\{x\in L(\Gamma):\tau_0\left(x\lambda(g)^*\right)=\tau_0\left(x\lambda(\varphi(g))^*\right)~\forall g\in\Gamma\right\}.\end{align*}   
\end{definition}

We now show that $\mathcal{N}_{\varphi}$ is a von Neumann subalgebra. We first make a simple observation about the order of such an automorphism. 

\begin{lemma}
\thlabel{orderoftheautomorphism}
Let $\varphi\in\text{Aut}(\Gamma)$ be a non-trivial automorphism and $\mathcal{N}_{\varphi}$ as above. Let $g\in\Gamma\setminus\{e\}$ be such that $g\in\text{Supp}(x)$ for some $x\in\mathcal{N}_{\varphi}$. Then, the orbit of $g$ under $\varphi$ is finite.
\begin{proof}
Let $g\in\Gamma\setminus\{e\}$ be such that $g\in\text{Supp}(x)$ for some $x\in\mathcal{N}_{\varphi}$. Then, $c_g=\tau_0(x\lambda(g)^*)\ne 0$. From our assumption, we see that $c_g=\tau_0(x\lambda(\varphi^n(g))^*)$ for every $n\in\mathbb{N}$. Now, 
\begin{align*}\|x\|_2^2&=\tau_0(x^*x)\\&\ge \tau_0\left(\left(\sum_{n=1}^kc_g\lambda(\varphi^n(g))\right)^*\left(\sum_{n=1}^kc_g\lambda(\varphi^n(g))\right)\right)\\&=\sum_{n=1}^k|c_g|^2\\&=k|c_g|^2\end{align*}
The claim follows since elements in $L(\Gamma)$ must have finite 2-norm (i.e., $\|x\|_2^2 < \infty$).
\end{proof}
\end{lemma}

It is a classical fact that one can extend $\varphi\in \text{Aut}(\Gamma)$ to get an automorphism on $L(\Gamma)$ and then define the fixed point von Neumann subalgebra. We make this precise. 

\begin{prop}
\thlabel{subalgebraassociated}
Let $\Gamma$ be a group with a nontrivial automorphism $\varphi$, and $\mathcal{N}_{\varphi}$ as above. Then, $\mathcal{N}_{\varphi}$ is a von Neumann algebra.
\begin{proof}
Fix any $\varphi\in \text{Aut}(\Gamma)$. It induces a unitary $U$ on $\ell^2(\Gamma)$ defined by $U\delta_g=\delta_{\varphi(g)}$ for all $g\in \Gamma$, where $\{\delta_g: g\in \Gamma\}$ denotes the standard orthonormal basis for $\ell^2(\Gamma)$. Then it is easy to check that for any $a=\sum_gc_g\lambda(g)\in \mathbb{C}\Gamma\subseteq \mathbb{B}(\ell^2(\Gamma))$, we have $UaU^*=\sum_gc_g\lambda(\varphi(g))\in\mathbb{C}\Gamma$. Hence we actually have $UL(\Gamma)U^*=L(\Gamma)$. We define the induced automorphism, still denoted by $\varphi\in \text{Aut}(L(\Gamma))$, by $\varphi=Ad(U)$. Then we can check that $\mathcal{N}_{\varphi}=L(\Gamma)^{\varphi}=\{a\in L(\Gamma): UaU^*=\varphi(a)=a\}$, which is clearly a von Neumann subalgebra. 
\end{proof}
\end{prop}

\begin{cor}
\thlabel{notoftheformofasubgroup}
Let $\varphi\in\text{Aut}(\Gamma)$ be a non-trivial automorphism of finite order and $\mathcal{N}_{\varphi}$ the associated von Neumann algebra. Then, $\mathcal{N}_{\varphi}$ is not of the form $L(H)$ for any subgroup $H\le \Gamma$.
\end{cor}
\begin{proof}
Suppose $\mathcal{N}_{\varphi} = L(H)$ for some subgroup $H$. Then the canonical conditional expectation satisfies $\mathbb{E}_{L(H)}(\lambda_g) \in \{\lambda_g, 0\}$ for all $g \in \Gamma$. However, since $\varphi$ has finite order $k > 1$, the conditional expectation onto $\mathcal{N}_\varphi = L(\Gamma)^\varphi$ is given by $\mathbb{E}_{\mathcal{N}_\varphi}(x) = \frac{1}{k}\sum_{i=0}^{k-1}\varphi^i(x)$.
Pick any $g \in \Gamma$ such that $\varphi(g) \neq g$. Then $\mathbb{E}_{\mathcal{N}_\varphi}(\lambda_g) = \frac{1}{k}\sum_{i=0}^{k-1}\lambda_{\varphi^i(g)}$, which is a non-trivial convex combination of orthogonal generators. Thus $\mathbb{E}_{\mathcal{N}_\varphi}(\lambda_g) \notin \{\lambda_g, 0\}$, contradicting $\mathcal{N}_{\varphi} = L(H)$.
\end{proof}
We are now ready to prove the existence of exotic Uniformly Recurrent States. By forcing the positive definite functions to have fractional values across an entire orbit, we guarantee that the resulting dynamical limits cannot possibly represent subgroups.
\begin{theorem}
\thlabel{thm:exotic-uras}
There exist groups $\Gamma$ and Uniformly Recurrent States $\mathcal{H} \in \mathrm{URA}(L(\Gamma))$ such that no element of $\mathcal{H}$ is of the form $\phi_{L(H)}$ for any subgroup $H \le \Gamma$. Consequently, the collection of uniformly recurrent states contains exotic elements that do not arise from Uniformly Recurrent Subgroups.
\end{theorem}
\begin{proof}
Let $\Lambda$ be any non-trivial group (e.g., $\Lambda = \mathbb{F}_2$). Define $\Gamma = \Lambda \times \Lambda$, and let $\varphi \in \text{Aut}(\Gamma)$ be the coordinate-flip automorphism $\varphi(x, y) = (y, x)$. Clearly, $\varphi$ has order $k=2$. 

Let $\mathcal{M} = \mathcal{N}_\varphi = L(\Gamma)^\varphi$. By Proposition~\ref{prop:continuity-equivariance-of-phi}, the associated positive definite function $\phi_{\mathcal{M}} \in \mathrm{PD}_1(\Gamma)$ is given by $\phi_{\mathcal{M}}(g) = \|\mathbb{E}_{\mathcal{M}}(\lambda_g)\|_2^2$. Since $\mathbb{E}_{\mathcal{M}}(\lambda_g) = \frac{1}{2}(\lambda_g + \lambda_{\varphi(g)})$, we compute:
\[
\phi_{\mathcal{M}}(g) = 
\begin{cases} 
1 & \text{if } \varphi(g) = g \\ 
\frac{1}{4}\|\lambda_g\|_2^2 + \frac{1}{4}\|\lambda_{\varphi(g)}\|_2^2 = \frac{1}{2} & \text{if } \varphi(g) \neq g 
\end{cases}
\]
Thus, $\phi_{\mathcal{M}}(g) \in \{1/2, 1\}$ for all $g \in \Gamma$.

Choose an element $g_0 = (a, e) \in \Gamma$ where $a \in \Lambda \setminus \{e\}$. For any arbitrary conjugating element $s = (u, v) \in \Gamma$, we have:
\[
s^{-1} g_0 s = (u^{-1} a u, e).
\]
Applying the automorphism $\varphi$ yields:
\[
\varphi(s^{-1} g_0 s) = (e, u^{-1} a u).
\]
Because $a \neq e$, the conjugate $u^{-1} a u$ is also non-trivial. Therefore, $(u^{-1} a u, e) \neq (e, u^{-1} a u)$, meaning $\varphi(s^{-1} g_0 s) \neq s^{-1} g_0 s$ for \emph{all} $s \in \Gamma$.

Consequently, for the specific element $g_0 = (a,e)$, we have:
\[
(s \cdot \phi_{\mathcal{M}})(g_0) = \phi_{\mathcal{M}}(s^{-1} g_0 s) = \frac{1}{2} \quad \text{for all } s \in \Gamma.
\]

Now, let $\mathcal{H}$ be any Uniformly Recurrent Algebra contained in the pointwise orbit closure $\overline{\{s \cdot \phi_{\mathcal{M}} : s \in \Gamma\}}^{\,\text{ptwise}}$ (such an $\mathcal{H}$ exists by compactness and Zorn's Lemma). By definition, any $\psi \in \mathcal{H}$ is a pointwise limit of functions of the form $s_n \cdot \phi_{\mathcal{M}}$. Because $(s \cdot \phi_{\mathcal{M}})(g_0) = 1/2$ universally, it must be that:
\[
\psi(g_0) = \frac{1}{2} \quad \text{for all } \psi \in \mathcal{H}.
\]

However, if an element $\psi \in \mathcal{H}$ were to come from a subgroup $H \le \Gamma$, it would have the form $\psi = \phi_{L(H)} = \mathbf{1}_H$, which implies $\psi(g) \in \{0, 1\}$ for all $g \in \Gamma$. Since $1/2 \notin \{0, 1\}$, no such $\psi$ can correspond to a subgroup algebra. Thus, $\mathcal{H}$ consists entirely of exotic operator-algebraic states.
\end{proof}
\begin{example}[Exotic uniformly recurrent states in Free Groups]
Let $\Gamma = \mathbb{F}_2 = \langle a, b \rangle$ be the free group on two generators. Let $\varphi \in \mathrm{Aut}(\mathbb{F}_2)$ be the involution that swaps the generators: $\varphi(a) = b$ and $\varphi(b) = a$. Let $\mathcal{M} = L(\mathbb{F}_2)^\varphi$ be the corresponding fixed-point von Neumann subalgebra.

As shown previously, the associated positive definite function $\phi_{\mathcal{M}}$ evaluates to $1$ on fixed points and $1/2$ on elements moved by the automorphism. The conjugacy classes of $a$ and $b$ are strictly disjoint.

For any arbitrary conjugating element $s \in \mathbb{F}_2$, the element $s^{-1} a s$ is conjugate to $a$, while $\varphi(s^{-1} a s) = \varphi(s)^{-1} b \varphi(s)$ is conjugate to $b$. Therefore, $\varphi(s^{-1} a s) \neq s^{-1} a s$ for all $s \in \mathbb{F}_2$. This implies that 
\[
(s \cdot \phi_{\mathcal{M}})(a) = \phi_{\mathcal{M}}(s^{-1}as) = \frac{1}{2} \quad \text{for all } s \in \mathbb{F}_2.
\]

Consequently, any Uniformly Recurrent State $\mathcal{H}$ contained in the pointwise orbit closure of $\phi_{\mathcal{M}}$ will satisfy $\psi(a) = 1/2$ for all $\psi \in \mathcal{H}$. Since subgroup indicator functions only take values in $\{0, 1\}$, $\mathcal{H}$ is an exotic uniformly recurrent state in $L(\mathbb{F}_2)$ that does not arise from any Uniformly Recurrent Subgroup.
\end{example}

\begin{example}[Exotic Uniformly Recurrent States in Higher-Rank Lattices]
Let $\Gamma = SL_3(\mathbb{Z})$. Consider the inverse-transpose automorphism $\varphi \in \mathrm{Aut}(\Gamma)$ defined by $\varphi(A) = (A^T)^{-1}$. Since $\varphi^2 = \mathrm{id}$, $\varphi$ is an involution. Let $\mathcal{M} = L(SL_3(\mathbb{Z}))^\varphi$ be the corresponding fixed-point subalgebra.

We seek an element $g_0 \in SL_3(\mathbb{Z})$ such that $g_0$ and $\varphi(g_0)$ are not conjugate in $\Gamma$. Note that $\varphi(g_0) = (g_0^T)^{-1}$ is conjugate in $GL_3(\mathbb{R})$ to $g_0^{-1}$. Thus, to ensure they reside in disjoint conjugacy classes, it suffices to find a matrix $g_0 \in SL_3(\mathbb{Z})$ whose eigenvalues are not invariant under inversion.

Let us define the matrix:
\[
g_0 = \begin{pmatrix} 0 & 0 & 1 \\ 1 & 0 & 1 \\ 0 & 1 & 0 \end{pmatrix}.
\]
We compute $\det(g_0) = 1$, confirming $g_0 \in SL_3(\mathbb{Z})$. The characteristic polynomial of $g_0$ is:
\[
\det(xI - g_0) = \det \begin{pmatrix} x & 0 & -1 \\ -1 & x & -1 \\ 0 & -1 & x \end{pmatrix} = x(x^2 - 1) - 1(1) = x^3 - x - 1.
\]
The eigenvalues of $g_0^{-1}$ are the roots of the reversed polynomial, $x^3 + x^2 - 1$. Since $x^3 - x - 1 \neq x^3 + x^2 - 1$, the matrices $g_0$ and $g_0^{-1}$ (and thus $\varphi(g_0)$) possess distinct sets of eigenvalues. Consequently, $g_0$ and $\varphi(g_0)$ lie in strictly disjoint conjugacy classes.

For any $s \in SL_3(\mathbb{Z})$, the element $s^{-1} g_0 s$ is conjugate to $g_0$, while $$\varphi(s^{-1} g_0 s) = \varphi(s)^{-1} \varphi(g_0) \varphi(s)$$ is conjugate to $\varphi(g_0)$. Because their conjugacy classes are disjoint, $\varphi(s^{-1} g_0 s) \neq s^{-1} g_0 s$ for all $s \in SL_3(\mathbb{Z})$.

Thus, $(s \cdot \phi_{\mathcal{M}})(g_0) = 1/2$ universally for all $s \in SL_3(\mathbb{Z})$. Any Uniformly Recurrent State $\mathcal{H}$ in the orbit closure of $\phi_{\mathcal{M}}$ must satisfy $\psi(g_0) = 1/2$ for all $\psi \in \mathcal{H}$. Therefore, $L(SL_3(\mathbb{Z}))$ admits an exotic Uniformly Recurrent State that does not arise from any Uniformly Recurrent Subgroup.
\end{example}
We now exhibit three additional infinite families of exotic Uniformly Recurrent States. In each case, the proof follows exactly the pattern of Theorem~\ref{thm:exotic-uras} and the subsequent examples. The associated positive-definite function \(\phi_{\mathcal{M}}\) is constant on a suitable orbit at the value \(1/k\) (\(k=\operatorname{ord}(\varphi)\geq 2\)), which forces every point in any minimal invariant subset of the orbit closure to lie outside the image of \(\mathrm{URS}(\Gamma)\).

\begin{example}[Cyclic permutation of generators in free groups \(F_n\) (\(n\geq 3\))]
Let \(\Gamma = F_n = \langle a_1,\dots,a_n\rangle\) (\(n\geq 3\)) and let \(\varphi\in\mathrm{Aut}(\Gamma)\) be the order-\(n\) automorphism
\[
\varphi(a_i) = a_{i+1}\quad (i=1,\dots,n-1),\qquad \varphi(a_n)=a_1.
\]
Let \(\mathcal{M} = L(\Gamma)^\varphi\) be the fixed-point subalgebra and \(\phi_{\mathcal{M}}\) its positive-definite function. Then
\[
\mathbb{E}_{\mathcal{M}}(\lambda_g) = \frac{1}{n}\sum_{k=0}^{n-1}\lambda_{\varphi^k(g)},
\]
so
\[
\phi_{\mathcal{M}}(g) = 
\begin{cases}
1 & \text{if }\varphi^k(g)=g\text{ for all }k,\\
\frac{1}{n} & \text{if the orbit }\{\varphi^k(g)\}_{k=0}^{n-1}\text{ has size exactly }n.
\end{cases}
\]
Choose \(g_0 = a_1\). The abelianization map \(\Gamma\to\mathbb{Z}^n\) sends \(a_1\) to the first standard basis vector \(e_1\). The images of \(\varphi^k(g_0)\) are exactly the distinct basis vectors \(e_{k+1}\). Hence, the conjugacy classes \([a_1]\), \([\varphi(a_1)]= [a_2]\), \ldots, \([a_n]\) are pairwise disjoint (they lie in different cosets of the commutator subgroup).  

For any \(s\in\Gamma\) the conjugate \(s^{-1}g_0 s\) is still in the class \([a_1]\), while \(\varphi(s^{-1}g_0 s)\) is in \([a_2]\). Thus \(\varphi(s^{-1}g_0 s)\neq s^{-1}g_0 s\) for all \(s\), and therefore
\[
(s\cdot\phi_{\mathcal{M}})(g_0) = \phi_{\mathcal{M}}(s^{-1}g_0 s) = \frac{1}{n}\qquad\forall s\in\Gamma.
\]
Any Uniformly Recurrent State \(\mathcal{H}\) contained in the (compact) pointwise orbit closure of \(\phi_{\mathcal{M}}\) satisfies \(\psi(g_0)=\frac{1}{n}\) for every \(\psi\in\mathcal{H}\). Since \(\frac{1}{n}\notin\{0,1\}\), no \(\psi\) can be of the form \(\mathbf{1}_H\), so \(\mathcal{H}\) is exotic.
\end{example}

\begin{example}[Permutation automorphisms on direct products of i.c.c.\ groups]
Let \(\Lambda\) be any non-trivial i.c.c.\ group and \(k\geq 2\). Set \(\Gamma = \Lambda^k\) (direct product) and let \(\varphi\) be the cyclic permutation
\[
\varphi(x_1,\dots,x_k) = (x_k,x_1,\dots,x_{k-1}).
\]
Again \(\varphi\) has order \(k\). Pick \(g_0 = (a,e,\dots,e)\) with \(a\in\Lambda\setminus\{e\}\). The abelianization argument (or the support of the word-length function) shows that the conjugacy classes of \(g_0\) and \(\varphi(g_0)\) are disjoint. The same constant-value argument yields \(\phi_{\mathcal{M}}(g_0) \equiv \frac{1}{k}\) on the entire \(\Gamma\)-orbit, hence exotic Uniformly Recurrent States.

This construction works even when \(\Lambda\) itself has a trivial outer automorphism group (e.g., certain hyperbolic groups or higher-rank lattices).
\end{example}

\begin{remark}
The same argument works whenever \(\Gamma\) admits a finite-order automorphism \(\varphi\) such that there exists \(g\in\Gamma\) whose conjugacy class is disjoint from the class of \(\varphi(g)\).
\end{remark}
\bibliographystyle{amsalpha}
\bibliography{confined}
\end{document}